%% file: almost_graphical_mcf.tex
\documentclass[pdftex,a4paper,12pt]{article}
\usepackage{amsmath}
\usepackage{amssymb}
\usepackage{amsthm}
\usepackage{mathrsfs}
\usepackage{makeidx}
\usepackage{comment}
\usepackage{authblk}
\usepackage[nottoc]{tocbibind}
\usepackage{hyperref}
\hypersetup{colorlinks=false, linkcolor=red, citecolor= red, urlcolor=red}
\allowdisplaybreaks[1]

\newcommand{\hball}[2]{\ensuremath{\mathbf{B}^1(#1,#2)}}
\newcommand{\chball}[1]{\ensuremath{\mathbf{B}^1(0,#1)}}
\newcommand{\bchball}[1]{\ensuremath{\mathbf{B}^1(0,#1)}}

\newcommand{\mcf}{mean curvature flow}

\begin{document}

\newtheorem{mthm}{Theorem}
\newtheorem*{remark}{Remark}
\swapnumbers
\newtheorem{thm}{Theorem}[subsection]
\newtheorem{lem}[thm]{Lemma}
\newtheorem{prop}[thm]{Proposition}
\newtheorem{cor}[thm]{Corollary}
\newtheorem{secthm}{Theorem}[section]
\newtheorem{seclem}[secthm]{Lemma}
\newtheorem{secprop}[secthm]{Proposition}
\newtheorem{seccor}[secthm]{Corollary}
\theoremstyle{definition}
\newtheorem{defn}[thm]{Definition}
\newtheorem{exmp}[thm]{Example}
\newtheorem{rem}[thm]{Remark}
\newtheorem{secdefn}[secthm]{Definition}
\newtheorem{secexmp}[secthm]{Example}
\newtheorem{secrem}[secthm]{Remark}
\numberwithin{equation}{section}

\title{Almost graphical hypersurfaces become graphical under mean curvature flow}
\author{Ananda Lahiri}
\affil{
Max Planck Institute for Gravitational Physics\\
(Albert Einstein Institute)
}
\maketitle

\pagestyle{plain}
\setcounter{page}{1}
%
%
\begin{abstract}
Consider a mean curvature flow of hypersurfaces in Euclidean space, 
that is initially graphical inside a cylinder.
There exists a period of time during which
the flow is graphical inside the cylinder of half the radius.
Here we prove a lower bound on this period depending on the Lipschitz-constant 
of the initial graphical representation.
This is used to deal with a mean curvature flow that lies inside a slab
and is initially graphical inside a cylinder except for a small set.
We show that such a flow will become graphical inside the cylinder of half the radius.
The proofs are mainly based on White's regularity theorem.
\end{abstract}
\tableofcontents

\section{Introduction}
\input{./introduction.tex}

\section{Arbitrary co-dimension}
\label{higher_codim}
\input{./higher_codim.tex}

\section{Hypersurfaces}
\label{one_codim}
\input{./one_codim.tex}
\begin{appendix}
\section{Appendix}
\label{appendix}
\input{./appendix.tex}
\end{appendix}

\input{./bibliography.tex}
\vspace*{1cm}
\noindent
Ananda Lahiri\\
Max Planck Institute for Gravitational Physics\\
(Albert Einstein Institute)\\
Am M\"uhlenberg 1, 14476 Potsdam, Germany\\
Ananda.Lahiri@aei.mpg.de
\end{document}

%% file: introduction.tex
\subsection{Overview}
\paragraph{Setting}
Let $t_1\in\mathbb{R}$ and $t_2\in (t_1,\infty]$.
Consider a (properly embedded) mean curvature flow $(M_t)_{t\in [t_1,t_2)}$
in some open set $U\subset\mathbb{R}^{\mathbf{n}+\mathbf{k}}$.
By this we mean the following:
There exist a submanifold $M$ of $\mathbb{R}^{\mathbf{n}+\mathbf{k}}$
and a smooth function $\Psi:[t_1,t_2)\times M\to\mathbb{R}^{\mathbf{n}+\mathbf{k}}$ 
such that $\Psi_t:=\Psi(t,\cdot)$ is 
an embedding with $M_t=\Psi_t[M]$ for all $t\in [t_1,t_2)$. 
Also $\Psi$ satisfies
\begin{align}
\label{meancurvflowa}
\partial_t\Psi(t,p)=\mathbf{H}(M_t,\Psi(t,p))
\end{align}
for all $(t,p)\in (t_1,t_2)\times M$,
where $\mathbf{H}$ denotes the mean curvature vector.
Furthermore we assume there exists an open set $V\subset\mathbb{R}^{\mathbf{n}+\mathbf{k}}$
with $\overline{U}\subset V$, $\mathscr{H}^{\mathbf{n}}(V\cap M_t)<\infty$ and 
$\partial M_t\cap V=\emptyset$ for all $t\in [t_1,t_2)$.

Consider open sets $\Omega_1\subset\mathbb{R}^{\mathbf{n}}$, $\Omega_2\subset\mathbb{R}^{\mathbf{k}}$.
The manifold $M_t$ is called graphical inside the cylinder $\Omega_1\times \Omega_2$,
if there exists an $f_t:\Omega_1\to\mathbb{R}^{\mathbf{k}}$
such that
\begin{align}
\label{localgraph}
M_{t}\cap (\Omega_1\times \Omega_2)
=
\mathrm{graph}(f_t).
\end{align}
Let $t_0\in [t_1,t_2)$, $x_0\in M_{t_0}$ and $\epsilon\in (0,1)$.
One can choose coordinates such that $x_0$ is the origin and the tangent space is $\mathbb{R}^{\mathbf{n}}\times\{0\}^{\mathbf{k}}$.
Parametrizing over the tangent space yields an $r>0$
such that $M_t$ is graphical inside the open cylinder $\mathbf{B}^{\mathbf{n}}(0,2r)\times\mathbf{B}^{\mathbf{k}}(0,2r)$ 
for all $t\in [t_0-r,t_0+r]\cap[t_1,t_2)$
and some parametrization with $\sup|f_{t_0}|\leq\epsilon r$, $\sup|Df_{t_0}|\leq\epsilon$.
In particular \eqref{localgraph} always holds locally for small enough $\Omega_1,\Omega_2$.
Note that lower bounds on $r$ can be related to curvature estimates on $(M_t)$.

\paragraph{Problem 1}
Consider a mean curvature flow $(M_t)_{t\in [0,\tau)}$.
Assume $M_0$ is graphical inside the open cylinder
$\mathbf{B}^{\mathbf{n}}(0,2)\times\mathbf{B}^{\mathbf{k}}(0,2)$ 
for some graphical representation $f_0$.
By continuity in $t$,
$M_t$ will remain graphical inside $\mathbf{B}^{\mathbf{n}}(0,1)\times\mathbf{B}^{\mathbf{k}}(0,1)$ for $t\in [0,\kappa)$
for some $\kappa>0$.
We are interested in lower bounds on $\kappa$
that only depend on $\mathbf{n}$, $\mathbf{k}$ and $\mathrm{lip}(f_0)$.

\paragraph{Problem 2}
Again consider a mean curvature flow $(M_t)_{t\in [0,\tau)}$.
Assume $M_0$ is almost graphical inside the open cylinder
$\mathbf{B}^{\mathbf{n}}(0,2)\times\mathbf{B}^{\mathbf{k}}(0,2)$
and also lies in a narrow slab $M_0\subset\mathbb{R}^{\mathbf{n}}\times\mathbf{B}^{\mathbf{k}}(0,\gamma)$.
By almost graphical we mean that there exist an $E\subset\mathbf{B}^{\mathbf{n}}(0,2)$
and an $f_0:\mathbf{B}^{\mathbf{n}}(0,2)\setminus E\to\mathbb{R}^{\mathbf{k}}$ such that
\begin{align*}
M_{0}\cap\left((\mathbf{B}^{\mathbf{n}}(0,2)\setminus E)\times\mathbf{B}^{\mathbf{k}}(0,2)\right)
=
\mathrm{graph}(f_0)
\end{align*}
and 
$E$ is small in some sense.
We are looking for conditions on $M_0$ and $E$ such that 
$M_t$ is graphical inside 
$\mathbf{B}^{\mathbf{n}}(0,1)\times\mathbf{B}^{\mathbf{k}}(0,1)$ 
for $t\in [\epsilon,\kappa)$ for some $0<\epsilon<\kappa\leq\tau$.
\\ \\
Note that if $\mathrm{lip}(f_0)$ is small,
there already exist satisfying statements for both problems, see below.
Hence we are especially interested in the case $\mathrm{lip}(f_0)\in [1,\infty)$.

\paragraph{History and known results}
The mean curvature flow was introduced by Brakke \cite{brakke} 
in the setting of geometric measure theory.
Starting with Huisken's work \cite{huisken1} the smooth mean curvature flow
came more into focus.
Graphical mean curvature flow of hypersurfaces,
was studied by Ecker and Huisken in \cite{eckerh2}, \cite{eckerh1}.
Among other things, they proved a local gradient bound \cite[2.1]{eckerh1}
and a local curvature estimate \cite[3.1, 3.2]{eckerh1},
under the assumption that the flow is graphical over a period of time.
These estimates will be of importance for our results.
Wang \cite{wang} generalised the estimates from Ecker and Huisken \cite[2.1, 3.1]{eckerh1} 
to higher co-dimension in \cite[3.2, 4.1]{wang},
where one additionally has to assume that the Lipschitz-constant of the graphical representation is small.
Wang observed, that such an extra assumption should be necessary 
in view of the minimal cone constructed by Lawson and Ossermann \cite[7.1]{lawsonO}.

The local regularity theorem by White \cite{white}
yields curvature bounds in case the Gaussian density ratios are close to one.
In particular it can be used for initially not graphical mean curvature flows.
Ilmanen, Neves and Schulze 
used White's regularity theorem to solve problem 1
if $\mathrm{lip}(f_0)\leq l_0(\mathbf{n},\mathbf{k})$, see \cite[1.5]{ilnesch}.
Changing the proof of \cite[1.5]{ilnesch} a bit yields a solution to problem 2,
if one assumes slab-height $\gamma$, $\mathrm{lip}(f_0)$ and $\mathscr{H}^{\mathbf{n}}(M_0\cap (E\times\mathbf{B}^{\mathbf{k}}(0,2)))$ 
to be bounded by some $\gamma_0(\epsilon,\mathbf{n},\mathbf{k})$.

Curvature bounds by Chen and Yin yield a solution to problem 1,
if $\sup|D^2f_0|\leq K$ and $\sup|Df_0|\leq l_1(K,\mathbf{n},\mathbf{k})$, see \cite[7.5]{chenyin}.
Moreover curvature bounds by Brendle and Huisken yield a solution to problem 1,
if $\|f_0\|_{C^4}\leq\beta_0$ for some constant $\beta_0$, see \cite[2.2]{brendleh}.

\paragraph{Results of the present article}
In case the co-dimension is one, 
we will answer problems 1 and 2 without assuming small gradient of $f_0$.
Referring to problem 1 we prove the following:
%
%
\begin{mthm}
\label{mainresult2}
For every $L\in (0,\infty)$ exists an $\kappa_L=\kappa_L(L,\mathbf{n})$
such that the following holds:
Let $(M_t)_{t\in [0,\tau)}$ be a mean curvature flow in $\mathbf{B}^{\mathbf{n}}(0,2)\times\chball{2}$.
Suppose there exists an $f_0:\mathbf{B}^{\mathbf{n}}(0,2)\to\mathbb{R}$
with $\mathrm{lip}(f_0)\leq L$ and $\sup|f_0|\leq\frac{1}{2}$
\begin{align*}
M_{0}\cap\mathbf{B}^{\mathbf{n}}(0,2)\times\chball{2}
=\mathrm{graph}(f_0)
\end{align*}
Set $I:=(0,\kappa_L)\cap (0,\tau)$.
Then there exists a $g:I\times\mathbf{B}^{\mathbf{n}}(0,1)\to\mathbb{R}$
with 
\begin{align*}
M_{t}\cap\mathbf{B}^{\mathbf{n}}(0,1)\times\chball{1}
=\mathrm{graph}(g(t,\cdot))
\end{align*}
for all $t\in I$.
See also Theorem \ref{stayinggraphicalthm} for a more general version.
\end{mthm}
\begin{remark}
\begin{enumerate}
\item
The local estimates by Ecker and Huiksen \cite[2.1, 3.2(ii)]{eckerh1} yield estimates on $|Dg|$ and $|D^2g|$.
\item
See Example \eqref{shrinkingsquareexmp} for a mean curvature flow that is initially graphical inside
$\mathbf{B}^{\mathbf{n}}(0,2)\times\chball{2}$ with gradient $0$,
but becomes non-graphical inside $\mathbf{B}^{\mathbf{n}}(0,2)\times\chball{2}$ arbitrarily fast.
Also the flow becomes non-graphical inside $\mathbf{B}^{\mathbf{n}}(0,1)\times\chball{1}$ after finite time.
\item
It is not clear if the dependency of $\kappa_L$ on $L$ is necessary.
\item
Also it is not clear, whether this result can be adopted to higher co-dimension.
In particular the cone constructed by Lawson and Ossermann \cite[7.1]{lawsonO},
has to be taken into account.
\end{enumerate}
\end{remark}

Referring to problem 2 we prove the following:
%
%
\begin{mthm}
\label{mainresult4}
There exists a $\kappa_2=\kappa_2(\mathbf{n})$ 
and for all $L\in (0,\infty)$, $\epsilon\in (0,\kappa_2)$ 
exists a $\gamma_2=\gamma_2(L,\epsilon,\mathbf{n})$
such that the following holds:
Let $(M_t)_{t\in [0,\tau)}$ be a mean curvature flow in $\mathbf{B}^{\mathbf{n}+1}(0,3)$.
Assume
\begin{align*}
M_{0}\cap\mathbf{B}^{\mathbf{n}+1}(0,3)
\subset \mathbb{R}^{\mathbf{n}}\times\chball{\gamma_2}.
\end{align*}
Suppose there exist a closed subset $E\subset\overline{\mathbf{B}^{\mathbf{n}-1}(0,2)}\times\chball{\gamma_2}$
and a function $f_0:\mathbf{B}^{\mathbf{n}}(0,2)\setminus E\to\mathbb{R}$ with $\mathrm{lip}(f_0)\leq L$
such that
\begin{align*}
M_{0}\cap\left((\mathbf{B}^{\mathbf{n}}(0,2)\setminus E)\times\chball{2}\right)
&=\mathrm{graph}(f_0),
\\
\mathscr{H}^{\mathbf{n}}\left(M_{0}\cap\left(E\times\chball{2}\right)\right)
&\leq\gamma_2.
\end{align*}
Set $I:=[\epsilon,\kappa_2)\cap (0,\tau)$.
Then there exists a 
$g:I\times\mathbf{B}^{\mathbf{n}}(0,1)\to\mathbb{R}$
with
\begin{align*}
M_{t}\cap\left(\mathbf{B}^{\mathbf{n}}(0,1)\times\chball{1}\right)
=\mathrm{graph}(g(t,\cdot))
\end{align*}
for all $t\in I$.
See also Theorem \ref{becominggraphthm} for a more general version,
which includes estimates for $|g|$, $|Dg|$ and $|D^2g|$.
\end{mthm}
\begin{remark}
As for Theorem \ref{mainresult2} it is not clear,
if the dependency of $\gamma_2$ on $L$ is necessary
and whether the result can be adopted to higher co-dimension.
\end{remark}

\paragraph{Outline of the proof of Theorem \ref{mainresult2}}
First note that for arbitrary $\delta_0$ we remain graphical inside 
$\mathbf{B}^{\mathbf{n}}(0,2-\delta_0)\times\bchball{2-\delta_0}$ 
for some period of time $[0,s_0]$.
Let $c_L$ denote values that only depend on $L$ and $\mathbf{n}$.
Suppose we have graphical representation inside 
$\mathbf{B}^{\mathbf{n}}(0,r)\times\chball{r}$ 
for some period of time $[0,s]$.
Using the curvature bound by Ecker and Huisken \cite[3.2(ii)]{eckerh1} 
we obtain that inside $\mathbf{B}^{\mathbf{n}}(0,r-\sqrt{s})\times\bchball{r-\sqrt{s}}$
the curvature of $M_s$ is bounded by $(c_L\sqrt{s})^{-1}$.
This lets us cover $M_s$ with balls of radius $c_Ll_0\sqrt{s}$,
such that we can parametrize over the tangent space
inside each of this balls with a parametrization with gradient smaller than $l_0$.
Here $l_0$ denotes a constant we obtain from the regularity result 
by Ilmanen, Neves and Schulze \cite[1.5]{ilnesch}.
Hence we can use \cite[1.5]{ilnesch} in each of the balls,
to maintain the graphical representation over the tangent spaces 
for a period of time $[s,(1+c_L)s]$.
Combining these graphical representations we see that $M_t$ is graphical inside
$\mathbf{B}^{\mathbf{n}}(0,r-\sqrt{s})\times\bchball{r-\sqrt{s}}$ for all $t\in[s,(1+c_L)s]$. 
Thus we can maintain graphical representation inside a shrinking cylinder,
where the shrinking is controlled by $L$ and $\mathbf{n}$.
This yields the result.
Note that in the process we use that the gradient is uniformly bounded by $4L$,
which is provided by estimate \cite[2.1]{eckerh1} from Ecker and Huisken.

\paragraph{Outline of the proof of Theorem \ref{mainresult4}}
Using Theorem \ref{mainresult2} we see that for some small $s_1\in (0,\epsilon)$
the manifold $M_{s_1}$ is still almost graphical inside 
$\mathbf{B}^{\mathbf{n}}(0,\frac{3}{2})\times\chball{\frac{3}{2}}$,
where the possibly non-graphical part has increased to some $\tilde{E}\times\chball{\frac{3}{2}}$.
By our initial assumptions we can bound
$\mathscr{H}^{\mathbf{n}}(M_{s_1}\cap(\tilde{E}\times\chball{\frac{3}{2}})$
in terms of $\gamma_2$.
Also, in view of the slab condition,
using the curvature bound by Ecker and Huisken \cite[3.2(ii)]{eckerh1} yields,
that the gradient on the graphical part of $M_{s_1}$ is bounded in terms of $\gamma_2$ as well.
Thus choosing $\gamma_2$ small enough,
we are almost in the situation to use 
the regularity result by Ilmanen, Neves and Schulze \cite[1.5]{ilnesch}.
Modifying the proof of \cite[1.5]{ilnesch} a bit,
yields the result.

\paragraph{Organisation of the paper}
In section \ref{notation} the notation is introduced.
This is followed by recalling some basic properties 
and an application of White's regularity theorem in section \ref{localregularity}.
Then we show, in section \ref{smallgradient}, that the work of Ilmanen, Neves and Schulze \cite{ilnesch}
already solves problems 1 and 2 in the case of small gradient.
This also yields a solution to problem 1 for bounded curvature,
presented in section \eqref{smallcurve}.
Afterwards we restrict ourselves to the case of one co-dimension.
In section \ref{stayinggraphical} we prove Theorem \ref{mainresult2},
which solves problem 1 for bounded gradient.
Finally in section \ref{becomegraph}
we prove Theorem \ref{mainresult4},
which solves problem 2 for bounded gradient.

%
%
%
%
%
%
\subsection{Notation and definitions}
\label{notation}
We set $\mathbb{R}^{+}:=\{x\in \mathbb{R},x\geq 0\}$,
$\mathbb{N}:=\{1,2,3,\ldots\}$
and $(a)_+:=\max\{a,0\}$ for $a\in\mathbb{R}$.
We fix $\mathbf{n},\mathbf{k}\in\mathbb{N}$.
Quantities that only depend on $\mathbf{n}$ and/or $\mathbf{k}$ are considered constant.
Such a constant may be denoted by $C$ or $c$,
in particular the value of $C$ and $c$ may change in each line.
We denote the canonical basis of $\mathbb{R}^{\mathbf{n}+\mathbf{k}}$ and $\mathbb{R}^{\mathbf{n}}$
by $(\mathbf{e}_i)_{1\leq i\leq\mathbf{n}+\mathbf{k}}$ and  $(\hat{\mathbf{e}}_i)_{1\leq i\leq\mathbf{n}}$ respectively.

Let $n,k\in\mathbb{N}$. 
Let $A:\mathbb{R}^n\to\mathbb{R}^k$ be linear.
We denote by $A^*$ the adjoint of $A$
and write $A(v)=A v$.
We consider the following norms
\begin{align*}
|A|^2&:=\mathrm{trace}(A^*A)
\;\;\;
\|A\|:=\sup_{v\in \mathbb{R}^n\setminus\{0\}}|A(v)||v|^{-1}.
\end{align*}

Let $T$ be an $n$-dimensional subspace of $\mathbb{R}^{n+k}$.
Set $T^{\bot}=\{x\in\mathbb{R}^{n+k}:\;x\cdot v=0\;\forall v\in T\}$.
By $T_{\natural}:\mathbb{R}^{n+k}\to T$ we denote the projection onto $T$,
i.e. if $T=\mathrm{span}(\tau_i)_{1\leq i\leq n}$, $T^{\bot}=\mathrm{span}(\nu_j)_{1\leq j\leq k}$
then $T_{\natural}$ is given by $T_{\natural}\tau_i=\tau_i$ and $T_{\natural}\nu_j=0$ for all $1\leq i\leq n, 1\leq j\leq k$.

For $R\in (0,\infty)$ and $b\in\mathbb{R}^n$ we set
\begin{align*}
\mathbf{B}^n(b,R):=\left\{x\in\mathbb{R}^{n}:|x-b|<R\right\},
\;\;\;\mathbf{B}^n(\hat{a},0):=\emptyset,
\;\;\;\mathbf{B}^n(\hat{a},\infty):=\mathbb{R}^{n},
\end{align*}
For $r,h\in [0,\infty]$ and 
$a=(\hat{a},\tilde{a})\in\mathbb{R}^{\mathbf{n}}\times\mathbb{R}^{\mathbf{k}}$ 
set
\begin{align*}
\mathbf{C}(a,r,h):=\mathbf{B}^{\mathbf{n}}(\hat{a},r)\times\mathbf{B}^{\mathbf{k}}(\tilde{a},h),
\;\;\;
\mathbf{B}(a,r):=\mathbf{B}^{\mathbf{n}+\mathbf{k}}(a,r).
\end{align*}

Consider an open subset $\Omega\subset\mathbb{R}^{n}$ 
and $t_1,t_2\in\mathbb{R}$, $t_1<t_2$.
For a function $\phi:(t_1,t_2)\times\Omega\to\mathbb{R}$
we denote by $\partial_t\phi$ the partial derivative in time (in $(t_1,t_2)$),
by $D_i\phi$ the partial derivatives in space $1\leq i\leq n$
and set $D\phi:=(D_1\phi,\ldots,D_n\phi)$.
%
%
%

%
%
%
%
\paragraph{Submanifolds of $\mathbb{R}^{\mathbf{n}+\mathbf{k}}$}
Let $M$ denote an ($\mathbf{n}$-dimensional, $\mathcal{C}^2$-regular) submanifold of $\mathbb{R}^{\mathbf{n}+\mathbf{k}}$.
By this we mean that for each $y\in M$ there are open sets $U,V\in\mathbb{R}^{\mathbf{n}+\mathbf{k}}$
with $y\in U$, $0\in V$ and a diffeomorphism $\psi\in\mathcal{C}^{2}(U,V)$ such that $\psi(0)=y$ and
\begin{align*}
\psi(\Omega)=M\cap U,
\;\;\;
\Omega:=V\cap\mathbb{R}^{\mathbf{n}}\times\{0\}^{\mathbf{k}}.
\end{align*}

For $\hat{p}\in\Omega$ and $x=\psi(\hat{p})$ we define
\begin{align*}
\tau_i(\hat{p})&:=D_i\psi(\hat{p})
\end{align*}
for $1\leq i\leq\mathbf{n}$.
The $\tau_i$ are linearly independent vectors spaning the tangent space
$\mathbf{T}(M,x)=\mathbf{T}_xM:=\mathrm{span}(\tau_i(\hat{p}))_{1\leq i\leq \mathbf{n}}$.
The first fundamental form $g(\hat{p})\in \mathbb{R}^{\mathbf{n}\times\mathbf{n}}$ is defined by
\begin{align*}
g_{ij}(\hat{p})&:=\tau_i(\hat{p})\cdot\tau_j(\hat{p}).
\end{align*}
The inverse $\left(g(\hat{p})\right)^{-1}$ always exists 
and its components will be denoted by $g^{ij}(\hat{p})$.
%
%
For $\phi\in\mathcal{C}^1(U,\mathbb{R})$ and $X\in\mathcal{C}^1(U,\mathbb{R}^{\mathbf{n}+\mathbf{k}})$  we define
\begin{align*}
\nabla_i^{M}\phi(x)
&:=\sum_{j=1}^{\mathbf{n}}g^{ij}D_i(\phi\circ\psi)\Big|_{\hat{p}}\cdot\tau_j(\hat{p}),
\;\;\;
\nabla^{M}\phi(x)
:=\sum_{i=1}^{\mathbf{n}}\nabla_i^{M}\phi(x)\tau_i(\hat{p})
\\
\mathrm{div}_M X(x)&:=\sum_{i=1}^{\mathbf{n}+\mathbf{k}}\left(\nabla^{M}(\mathbf{e}_i\cdot X(x))\right)\cdot\mathbf{e}_i
=\sum_{i=1}^{\mathbf{n}+\mathbf{k}}\left(DX(x)\tau_i(\hat{p})\right)\cdot\tau_i(\hat{p})
\end{align*}
We define the second fundamental form 
$A(\hat{p})\in\mathbb{R}^{\mathbf{n}\times\mathbf{n}\times(\mathbf{n}+\mathbf{k})}$ by
\begin{align*}
A_{ij}(\hat{p})
&:=\left(\mathbf{T}(M,x)^{\bot}\right)_{\natural}D_iD_j\psi(\hat{p})\in\mathbb{R}^{\mathbf{n}+\mathbf{k}}.
\end{align*}
Moreover we define the mean curvature vector
\begin{align*}
\mathbf{H}(M,x):=\sum_{i,j=1}^{\mathbf{n}}g^{ij}(\hat{p})A_{ij}(\hat{p})
\end{align*}
and the norm
\begin{align*}
\left|\mathbf{A}(M,x)\right|^2
:=\sum_{i,j,\imath,\jmath=1}^{\mathbf{n}}
g^{i\imath}(\hat{p})g^{j \jmath}(\hat{p})A_{ij}(\hat{p})\cdot A_{\imath\jmath}(\hat{p}).
\end{align*}
Note that these two quantities are both independent of the choice of $\psi$.

$M$ induces a Borel measure on $\mathbb{R}^{\mathbf{n}+\mathbf{k}}$ via
$\mu_M(A):=\mathscr{H}^{\mathbf{n}}\left(M\cap A\right)$
for all $A\subset\mathbb{R}^{\mathbf{n}+\mathbf{k}}$.
For $I\subset J\subset\mathbb{R}$, $A\subset\mathbb{R}^{\mathbf{n}+\mathbf{k}}$,
a family $(M_t)_{t\in I}$ of submanifolds 
and a function $\phi:J\times A\to\mathbb{R}^N$
we set $\int_{M_t\cap A}\phi :=\int_{A}\phi(t,x)d\mu_{M_t}(x)$,
supposed this expression exists.
We will often just write $\mu$ or $\mu_t$ instead of $\mu_M$ or $\mu_{M_t}$.

%
%
\paragraph{Graphical submanifolds}
Let $M$ be a submanifold of $\mathbb{R}^{\mathbf{n}+\mathbf{k}}$.
Consider the special case
\begin{align*}
M\cap\mathbf{C}(y,R,\Gamma)=\mathrm{graph}(f)
\end{align*}
for $y=(\hat{y},\tilde{y})\in M$, $R,\Gamma\in (0,\infty)$
and some $f\in\mathcal{C}^{2}\left(\mathbf{B}^{\mathbf{n}}(\hat{y},R),\mathbb{R}^{\mathbf{k}}\right)$.
The first derivative of $f$ is bounded by the tilt of the tangent space
$\mathrm{tilt}(M,y):=\left\|(\mathbb{R}^{\mathbf{n}}\times\{0\}^{\mathbf{k}})_{\natural}-\mathbf{T}(M,y)_{\natural}\right\|^2$
in the following way (see \cite[8.9(5)]{allard})
\begin{align}
\label{tiltderivativerema}
\mathrm{tilt}(M,y)
\leq \|Df(\hat{y})\|^2
\leq \frac{\mathrm{tilt}(M,y)}{1-\mathrm{tilt}(M,y)}.
\end{align}
The second derivative of $f$ is related to the curvature $\mathbf{A}$ in the following way
\begin{align}
\label{secderivativerema}
c\frac{|D^2f(\hat{y})|^2}{\left(1+|Df(\hat{y})|^2\right)^{3}}
\leq \left|\mathbf{A}(M,y)\right|^2
\leq C|D^2f(\hat{y})|^2.
\end{align}
Consider $F\in\mathcal{C}^{2}\left(\mathbf{B}^{\mathbf{n}}(\hat{y},R),\mathbb{R}^{\mathbf{n}+\mathbf{k}}\right)$
given by $F(\hat{x}):=(\hat{x},f(\hat{x}))$.
We denote by $JF(\hat{x}):=\sqrt{DF(\hat{x})^TDF(\hat{x})}$ the Jacobian of $F$ at $\hat{x}$.
For a function $\phi:\mathbf{C}(y,R,\Gamma)\to\mathbb{R}^N$
we have
\begin{align*}
\int_{M\cap\mathbf{C}(y,R,\Gamma)}\phi 
=\int_{F^{-1}[M\cap\mathbf{C}(y,R,\Gamma)]}\phi(F(\hat{x}))JF(\hat{x})d\mathscr{L}^{\mathbf{n}}(\hat{x}),
\end{align*}
supposed this expression exists.

In the special case $\mathbf{k}=1$ the following holds:
The normal space $\mathbf{T}(M,y)^{\bot}$ is spanned by the normal vector
\begin{align}
\label{hypergraphnormal}
\nu(y)=\frac{(-Df(\hat{y}),1)}{\sqrt{1+|Df(\hat{y})|^2}}.
\end{align}
The tilt is given by
\begin{align}
\label{hypertilt}
\left\|(\mathbb{R}^{\mathbf{n}}\times\{0\})_{\natural}-\mathbf{T}(M,y)_{\natural}\right\|^2
=1-(\nu(y)\cdot\mathbf{e}_{\mathbf{n}+1})^2
=\frac{|Df(\hat{y}|^2}{1+|Df(\hat{y})|^2}.
\end{align}

\paragraph{Comments on the definition of mean curvature flow}
By the smoothness of $\Psi:[t_1,t_2)\times M\to\mathbb{R}^{\mathbf{n}+\mathbf{k}}$,
we mean that there exist $t_0\in (-\infty,t_1)$, $\Omega\subset\mathbb{R}^{\mathbf{n}+\mathbf{k}}$ open
and a $\bar{\Psi}\in\mathcal{C}^{\infty}((t_0,t_2)\times\Omega,\mathbb{R}^{\mathbf{n}+\mathbf{k}})$
such that $M\subset\Omega$ and $\bar{\Psi}|_{[t_1,t_2)\times M}\equiv \Psi$.
In particular all $M_t$ are smooth submanifolds of $\mathbb{R}^{\mathbf{n}+\mathbf{k}}$
and all graphical representations are smooth as well.

Actually we only use $\Psi$ is smooth on $(t_1,t_2)$,
$\Psi_{t_1}$ is locally Lipschitz (sections \ref{localregularity} and \ref{smallgradient}) 
or $\mathcal{C}^{2}$ (sections \ref{smallcurve}, \ref{stayinggraphical} and \ref{becomegraph})
and $\int_{M_{t_1}}\phi=\lim_{t\searrow t_1}\int_{M_t}\phi$ for all $\phi\in\mathcal{C}_c^{2}(\mathbb{R}^{\mathbf{n}+\mathbf{k}})$.

%% file: higher_codim.tex
%
%
%
%
%
%
%
%
%
\subsection{Local regularity}
\label{localregularity}
%
%
%
%
The following integrated version of the mean curvature flow equation \eqref{meancurvflowa} is from Brakke \cite[3.5]{brakke}.
For smooth mean curvature flow  this follows from the evolution equations by Huisken \cite[3.6]{huisken1} (see also Ecker \cite[4.6]{eckerb}).
%
%
\begin{prop}[{\cite[3.5]{brakke}}]
\label{timedepfctprop}
Consider an open, bounded subset $U\subset\mathbb{R}^{\mathbf{n}+\mathbf{k}}$
and let $t_1\in\mathbb{R}$, $t_2\in (t_1,\infty]$.
Let $(M_t)_{t\in [t_1,t_2)}$ be a \mcf\; in $U$
and $\phi\in\mathcal{C}^2\left((t_1,t_2)\times\mathbb{R}^{\mathbf{n}+\mathbf{k}},\mathbb{R}\right)$
with $\cup_{t\in (t_1,t_2)}M_t\cap\{\phi_t\neq 0\}\subset U$.
Then for all $s\in (t_1,t_2)$
\begin{align}
\label{intmeancurvflowa}
\begin{split}
\frac{d}{dt}\bigg|_{t=s}\int_{M_t}\phi_t
&=
\int_{M_t}\left(
\partial_t\Big|_{t=s}\phi_t+\mathrm{div}_{M_s}(D\phi_s)-|\mathbf{H}_s|^2\phi_s
\right),
\end{split}
\end{align}
where $\phi_t(x)=\phi(t,x)$.
\end{prop}

%
%
%
%
The above Proposition shows that if $\left(\partial_t-\mathrm{div}_{M_t}D\right)\phi\leq 0$,
the $M_t$ integral of $\phi$ is monotonically non-increasing.
If a test function $\eta$ satisfies $\left(\partial_t-\mathrm{div}_{M_t}D\right)\eta\leq 0$,
but does not have bounded support, 
we can use the following Proposition to localize $\eta$.
See also \cite[3.6]{brakke}, \cite[3.17]{eckerb} and \cite{eckerh2}.
\begin{prop}
\label{localisingprop}
Let $\varrho\in (0,\infty)$, $L\in [0,\infty)$, $t_1\in\mathbb{R}$, $t_2\in (t_1,\infty]$, $y_0\in\mathbb{R}^{\mathbf{n}+\mathbf{k}}$
and let $(M_t)_{t\in [t_1,t_2)}$ be a \mcf\; in $\mathbf{B}(y_0,\varrho)$.
Consider a function $\eta\in\mathcal{C}^{0,1}\left((t_1,t_2)\times\mathbf{B}(y_0,\varrho),[0,1]\right)
\cap\mathcal{C}^2\left(\{\eta>0\}\right)$
which satisfies
\begin{align}
\label{localisingpropa}
\left(\partial_t-\mathrm{div}_{M_t}D\right)\eta(t,x)
\leq 0,
\;\;\;
|\nabla^{M_t}\eta(t,x)|\leq L
\end{align}
for all $(t,x)\in (t_1,t_2)\times\mathbf{B}(y_0,\varrho)\cap\{\eta>0\}$ 
with $x\in M_t$.
Define $\Upsilon\in\mathcal{C}^{0,1}\left((t_1,t_2)\times\mathbf{B}(y_0,\varrho),[0,1]\right)$ by
\begin{align*}
\Upsilon(t,x):=\left(\left(\varrho^2-|x-y_0|^2\right)\eta(t,x)-(2\mathbf{n}+L\varrho)(t-t_1)\right)_+.
\end{align*}
Then for all $t\in (t_1,t_2)$
\begin{align}
\label{localisingpropc}
\frac{d}{dt}
\int_{\mathbb{R}^{\mathbf{n}+\mathbf{k}}}\Upsilon^3(t,x)d\mu_t(x)
\leq 0.
\end{align}
\end{prop}
%
%
\begin{proof}
We may assume $y_0=0$ and $t_1=0$.
Set $\mathscr{M}:=\{(t,x)\in (t_1,t_2)\times\mathbb{R}^{\mathbf{n}+\mathbf{k}}:x\in M_t\}$.
Note that $\Upsilon^3\in\mathcal{C}^2\left((t_1,t_2)\times\mathbf{B}(y_0,\varrho),[0,1]\right)$
and $\{\Upsilon>0\}\subset\{\eta>0\}\cap\mathbf{B}(0,\varrho)$.
Using $\mathrm{div}_{M_t}(x)=\mathbf{n}$, $\eta\leq 1$ and \eqref{localisingpropa}
we can estimate
\begin{align*}
\left(\partial_t-\mathrm{div}_{M_t}D\right)\Upsilon^3(t,x)
\leq 3\Upsilon^2\left(\varrho^2-|x|^2\right)\left(\partial_t-\mathrm{div}_{M_t}D\right)\eta(t,x)\leq 0
\end{align*}
on $\mathscr{M}\cap\{\Upsilon>0\}$.
Continuity of the derivatives of $\Upsilon^3$ implies 
that this estimate holds on all of $\mathscr{M}$.
Also we know $\cup_{t\in (t_1,t_2)}\{\Upsilon(t,\cdot)>0\}\subset\mathbf{B}(0,\varrho)$.
Using Proposition \ref{timedepfctprop} then establishes the result.
\end{proof}

%
%
%
%
%
In the very important special case $\eta\equiv 1$ and $L=0$
the $\Upsilon$ from above proposition will be called $\varphi$.
This test function already appears in \cite[ch.3]{brakke}.
\begin{defn}
\label{barrierfctdef}
Let $\rho\in (0,\infty)$, $x_0\in\mathbb{R}^{\mathbf{n}+\mathbf{k}}$, $t_0\in\mathbb{R}$ be fixed.
Define $\varphi_{\rho}\in\mathcal{C}^{0,1}\left(\mathbb{R}\times\mathbb{R}^{\mathbf{n}+\mathbf{k}}\right)$ by
\begin{align*}
\varphi_{\rho}(t,x)&:=\left(1-\rho^{-2}\left(|x|^2 + 2\mathbf{n}t\right)\right)_+,
\;\;\;
\varphi_{(t_0,x_0),\rho}(t,x):=\varphi_{\rho}(t-t_0,x-x_0).
\end{align*}
\end{defn}

%
%
%
%
Applying Proposition \ref{localisingprop}
directly yields a monotonicity formula for $\varphi_{\rho}$,
which implies an a-priori measure bound for balls.
\begin{cor}[{\cite[3.7]{brakke}},{\cite[4.9]{eckerb}}]
\label{barrierlem}
Let $\varrho\in (0,\infty)$, $s_1\in\mathbb{R}$, $s_2\in (s_1,\infty)$, $y_0\in\mathbb{R}^{\mathbf{n}+\mathbf{k}}$
and let $(M_t)_{t\in [s_1,s_2)}$ be a \mcf\; in $\mathbf{B}(y_0,\varrho)$.
Then
\begin{align*}
\frac{d}{dt}
\int_{\mathbb{R}^{\mathbf{n}+\mathbf{k}}}\varphi_{(s_1,y_0),\varrho}^3(t,x)d\mu_{t}(x)
\leq 0
\end{align*}
for all $t\in (s_1,s_2)$.
Also we have
\begin{align*}
\mathscr{H}^{\mathbf{n}}\left(M_t\cap\mathbf{B}(y_0,2^{-1}\varrho)\right)
\leq 8\mathscr{H}^{\mathbf{n}}\left(M_{s_1}\cap\mathbf{B}(y_0,\varrho)\right)
\end{align*}
for all $t\in [s_1,s_1+(8n)^{-1} \varrho^2)\cap[s_1,s_2)$.
\end{cor}

%
%
%
%
Proposition \ref{localisingprop} also implies that 
a local height bound is maintained.
Though the bound is increasing linear in time.
%
%
\begin{cor}
\label{heightboundcor}
There exists a $C\in (1,\infty)$ such that the following holds:
Let $R,r_0\in (0,\infty)$, $t_1\in\mathbb{R}$, $t_2\in (t_1,\infty)$,
$x_0\in\mathbb{R}^{\mathbf{n}+\mathbf{k}}$
and let $(M_t)_{t\in [t_1,t_2)}$ be a \mcf\; in $\mathbf{B}(x_0,2R)$.
Suppose
\begin{align*}
M_{t_1}\cap\mathbf{B}(x_0,2R)
\subset\mathbf{C}(x_0,2R,r_0).
\end{align*}
Then for all $t\in [t_1,t_2)$ and $r(t):=r_0+C(t-t_1)R^{-1}$ we have
\begin{align*}
M_{t}\cap\mathbf{C}(x_0,R,R)
\subset\mathbf{C}(x_0,R,r(t)).
\end{align*}
\end{cor}
%
%
\begin{proof}
We may assume $t_1=0$, $x_0=0$ and $R=1$.
Use Proposition \ref{localisingprop} with $\varrho=2$, $L=2^{-1}$ and
\begin{align*}
\eta(t,(\hat{x},\tilde{x})):=2^{-1}\left(|\tilde{x}|-r_0\right)_+,
\;\;\;
\zeta(x):=(4-|x|^2)_+.
\end{align*}
Then by \eqref{localisingpropc} we have
\begin{align*}
\int_{\mathbb{R}^{\mathbf{n}+\mathbf{k}}}\left(\zeta(x)\eta(t,(\hat{x},\tilde{x}))-Ct\right)_+d\mu_t
\leq \int_{\mathbb{R}^{\mathbf{n}+\mathbf{k}}}\left(\zeta(x)\eta(0,(\hat{x},\tilde{x}))\right)_+d\mu_0
\end{align*}
By assumption the right hand side is zero. 
For $x\in\mathbf{C}(0,1,1)$ we have $\zeta\geq 2$,
hence by definition of $\eta$ we obtain the result.
\end{proof}

%
%
%
%
%
%
%
%
%
%
%
%
The main ingredient to obtain graphical representation,
will be White's regularity theorem \cite{white}.
The version stated here follows from the proof by Ecker \cite[5.6]{eckerb}
and Huisken's monotonicity formula \cite[3.1]{huisken2}.
%
%
\begin{defn}
\label{heatkerneldef}
Let $x_0\in\mathbb{R}^{\mathbf{n}+\mathbf{k}}$, $t_0\in\mathbb{R}$ be fixed.
Define $\Phi_{(t_0,x_0)}\in\mathcal{C}^{\infty}
\left((-\infty,t_0)\times\mathbb{R}^{\mathbf{n}+\mathbf{k}}\right)$ by
\begin{align*}
\Phi_{(t_0,x_0)}(t,x)&:=\left(4\pi(t_0-t)\right)^{-\frac{\mathbf{n}}{2}}\mathrm{exp}\left(\frac{|x-x_0|^2}{4(t-t_0)}\right).
\end{align*}
\end{defn}
\begin{rem}
\label{heatkernelrem}
We have $\int_{T}\Phi_{(t_0,x_0)}(t,x)d\mathscr{H}^{\mathbf{n}}(x)=1$
for every $\mathbf{n}$-dimensional subspace $T$ of $\mathbb{R}^{\mathbf{n}+\mathbf{k}}$ and all $t\in (-\infty,t_0)$.
\end{rem}
%
%
\begin{thm}[{\cite{white}}]
\label{smoothregcor}
There exist $C\in (1,\infty)$ and $d_0\in (0,1)$ such that the following holds:
Let $r_0\in (0,\infty)$, $t_0\in\mathbb{R}$, $t_1\in (-\infty, t_0-8r_0^2]$, $t_2\in (t_0,\infty)$,
$x_0\in\mathbb{R}^{n+k}$, $R_0=C\sqrt{t_0-t_1}$
and let $(M_t)_{t\in [t_1,t_2)}$ be a \mcf\; in $\mathbf{B}(x_0,R_0)$.
Suppose $x_0\in M_{t_0}$ and
\begin{align}
\label{smoothregcora}
\int_{\mathbf{B}(x_0,R_0)}\Phi_{(s,x)}(t_1,y)d\mu_{t_1}(y)\leq 1 + d_0
\end{align}
for all $(s,x)\in [t_0-4r_0^2,t_0]\times\mathbf{B}(x_0,2r_0)$.
Then 
\begin{align}
\label{smoothregcorb}
\left|\mathbf{A}(M_t,x)\right|\leq Cr_0^{-1}
\end{align}
for all $t\in [t_0-r_0^2,t_0]$ and $x\in M_t\cap\mathbf{B}(x_0,r_0)$.
\end{thm}
\begin{proof}
We may assume $x_0=0$ and $t_0=0$.
Consider $\rho:=\frac{R_0}{4}$.
Note that $\rho=\frac{C}{4}\sqrt{-t_1}\geq 2r_0$, for $C$ large enough.
Consider arbitrary $(s,x)\in [-4r_0^2,0]\times\mathbf{B}(0,2r_0)$
and $t\in [s-4r_0^2,s)$.
We can estimate $2n(s-t_1)\leq\frac{8\mathbf{n}\rho^2}{C^2}$.
This yields $\mathrm{spt}\varphi_{(s,x),\rho}(t_1,\cdot)\subset\mathbf{B}(0,R_0)$ 
and $\sup|\varphi_{(s,x),\rho}(t_1,\cdot)|\leq 1+d_0$,
where we chose $C$ large enough depending on $d_0$.
Hence using \eqref{smoothregcora} we can estimate
\begin{align*}
\int_{M_{t_1}}\Phi_{(s,x)}\varphi_{(s,x),\rho}
\leq
(1+d_0)\int_{M_{t_1}\cap\mathbf{B}(0,R_0)}\Phi_{(s,x)}
\leq
(1+d_0)^2
\leq 
1+4d_0.
\end{align*}

Let $t\in [s-4r_0^2,s)$, in particular $t_1\leq t$.
Hence Huisken's monotonicity formula \cite[3.1]{huisken2}
(see also Ecker's localised version \cite[4.8]{eckerb})
yields
\begin{align*}
\int_{M_t}\Phi_{(s,x)}\varphi_{(s,x),\rho}
\leq
\int_{M_{t_1}}\Phi_{(s,x)}\varphi_{(s,x),\rho}
\leq 
1+4d_0.
\end{align*}
Now we can proceed as in \cite[5.6]{eckerb}.
\end{proof}

%
%
%
%
%
%
%
%
\subsection{Stay graphical for small gradient}
\label{smallgradient}
Here we give the proof of Theorem \ref{flatstaygraphthm},
which is the local version of Theorem 1.5 from
Ilmanen, Neves and Schulze \cite[1.5]{ilnesch}.
Basically this section is a slight variation of section 9 of \cite{ilnesch}.
In particular we point out Proposition \ref{flatbecomegraphprop},
which can be easily obtained from the proof of \cite[1.5]{ilnesch}.
For the convenience of the reader we include all the details.

%
%
%
%
\begin{thm}[{\cite[1.5]{ilnesch}}]
\label{flatstaygraphthm}
There exist $C\in (1,\infty)$ and $l_0\in (0,1)$ such that the following holds:
Let $\rho,\tau\in (0,\infty)$, $\delta\in (0,1]$, $l\in [0,l_0]$, 
let $(M_t)_{t\in [0,\tau)}$ be a \mcf\; in $\mathbf{C}(a,2\rho,2\rho)$
and let $a=(\hat{a},\tilde{a})\in M_{0}$.
Suppose there exists a function $f:\mathbf{B}^{\mathbf{n}}(\hat{a},2\rho)\to \mathbb{R}^{\mathbf{k}}$ 
with $\sup|Df|\leq l$ and
\begin{align}
\label{flatstaygraphthma} 
M_{0}\cap\mathbf{C}(a,2\rho,2\rho)
=\mathrm{graph}(f).
\end{align}
Set $I:=(0,l_0\rho^2)\cap (0,\tau)$.
Then there exists a $g:I\times\mathbf{B}^{\mathbf{n}}(\hat{a},\rho)\to\mathbb{R}^{\mathbf{k}}$
with
\begin{align}
\label{flatstaygraphthmb} 
M_{t}\cap\mathbf{C}(a,\rho,\rho)
=\mathrm{graph}(g(t,\cdot))
\end{align}
and
\begin{align}
\label{flatstaygraphthmc} 
\begin{split}
\sup|g(t,\cdot)-\tilde{a}|&\leq 2l\rho+C\rho^{-1}t,
\\
\sup|Dg(t,\cdot)|&\leq C\sqrt[4]{l+\rho^{-2}t},
\\
\sup|D^2g(t,\cdot)|&\leq Ct^{-\frac{1}{2}}
\end{split}
\end{align}
for all $t\in I$.
\end{thm}
%
%
%

%
%
%
%
\begin{prop}[{\cite[1.5]{ilnesch}}]
\label{flatbecomegraphprop}
There exist $C\in (1,\infty)$ and $\beta_0\in (0,1)$ such that the following holds:
Let $\beta\in (0,\beta_0)$, $\varrho\in (0,\infty)$,
$s_1\in\mathbb{R}$, $s_0:=s_1+\beta^2\varrho^2$, $s_2\in (s_0,\infty)$,
$z_0=(\hat{z}_0,\tilde{z}_0)\in\mathbb{R}^{\mathbf{n}+\mathbf{k}}$
and let $(M_t)_{t\in [s_1,s_2)}$ be a \mcf\; 
in $\mathbf{C}(z_0,4\varrho,4\varrho)$.
Set $J:=[s_0+\beta_0^2\varrho^2)\cap [s_0,s_2)$
Assume
\begin{align}
\label{flatbecomegraphpropa1} 
M_{s_1}\cap\mathbf{C}(z_0,4\varrho,4\varrho)
\subset\mathbf{C}(z_0,4\varrho,\beta^2\varrho),
\\
\label{flatbecomegraphpropa2} 
M_{s_0}\cap\mathbf{C}(z_0,\varrho,\varrho)
\neq\emptyset
\end{align}
Suppose there exist an open set $D_0\subset\mathbf{B}^{\mathbf{n}}(\hat{z}_0,4\varrho)$
and an $f_0:D_0\to\mathbb{R}^{\mathbf{k}}$ such that
\begin{align}
\label{flatbecomegraphpropb1} 
M_{s_1}\cap D_0\times\mathbf{B}^{\mathbf{k}}(\tilde{z}_0,\varrho)
=\mathrm{graph}(f_0),
\\
\label{flatbecomegraphpropb2} 
\sup|Df_0|\leq\beta_0,
\\
\label{flatbecomegraphpropb3} 
\mathscr{H}^{n}\left(M_{s_1}\cap S_0\right)
\leq\beta_0\beta^{{\mathbf{n}}}\varrho^{\mathbf{n}},
\end{align}
where $S_0:=(\mathbf{B}^{\mathbf{n}}(\hat{z}_0,4\varrho)\setminus D_0)\times\mathbf{B}^{\mathbf{k}}(\tilde{z}_0,\varrho)$.
Then there exists a function 
$g:J\times\mathbf{B}^{\mathbf{n}}(\hat{z}_0,\varrho)\to\mathbb{R}^{\mathbf{k}}$
with
\begin{align}
\label{flatbecomegraphpropc} 
M_{t}\cap\mathbf{C}(z_0,\varrho,\varrho)
=\mathrm{graph}(g(t,\cdot))
\end{align}
and
\begin{align}
\label{flatbecomegraphpropd}
\begin{split}
\sup|g(t,\cdot)-\tilde{z}_0|&\leq C\varrho^{-1}(t-s_1),
\\
\sup|Dg(t,\cdot)|&\leq C\sqrt[4]{\varrho^{-2}(t-s_1)},
\\
\sup|D^2g(t,\cdot)|&\leq C(t-s_1)^{-\frac{1}{2}}
\end{split}
\end{align}
for all $t\in J$.
\end{prop}
\begin{rem}
\label{flatbecomegraphrem}
\begin{enumerate}
\item
A similar result for Brakke flows can be found in the author's theis \cite[11.7]{lahiri}.
\item 
Assumption \eqref{flatbecomegraphpropa2} 
can be replaced by
\begin{align}
\label{flatbecomegraphpropb4} 
\mathbf{B}^{\mathbf{n}}(\hat{a}_0,\sqrt{\beta}\varrho)\subset D
\;\;\ \text{for some}\;\; \hat{a}_0\in \mathbf{B}^{\mathbf{n}}(\hat{z}_0,\varrho).
\end{align}
\item
If $M_0$ has bounded curvature and $\mathbf{n}\geq\mathbf{k}+1$,
assumption \eqref{flatbecomegraphpropb3} 
follows from $\mathbf{B}^{\mathbf{n}}(\hat{z}_0,4\varrho)\setminus D_0\subset\mathbf{B}^{\mathbf{n}}(\hat{y}_0,\beta\varrho)$,
see Lemma \ref{pcurveboundlem}.
\end{enumerate}
\end{rem}

\begin{proof}[Proof of Proposition \ref{flatbecomegraphprop}]
We may assume $s_1=0$ $z_0=0$ and $\varrho=1$.
The idea is, that the almost graphical representation
implies \eqref{smoothregcora}. 
Hence Theorem \ref{smoothregcor} yields curvature bounds.
Combining these bounds with the slab condition implies,
that the flow has to be a union of graphs.
Finally by continuity in time and assumption \eqref{flatbecomegraphpropa2} 
we see, that we actually have exactly one graph.

Let $t_0\in J$ be fixed but arbitrary.
In view of \eqref{flatbecomegraphpropa1},
Corollary \ref{heightboundcor}
with $R=2$, $r_0=\beta^2$ yields
\begin{align}
\label{flatbecomegraphprop11}
M_{t}\cap\mathbf{C}(0,2,2)
\subset\mathbf{C}(0,2,C_1t_0)
\end{align}
for all $t\in [0,t_0]$ for some constant $C_1\in (1,\infty)$.

Consider $F_0(\hat{v}):=(\hat{v},f_0(\hat{v}))$.
Assumptions \eqref{flatbecomegraphpropb1}-\eqref{flatbecomegraphpropb3} imply
\begin{align}
\label{flatbecomegraphprop23}
\int_{M_{0}\cap\mathbf{B}(0,4)}\phi
\leq
(1+C\beta_0)\int_{D_0}\phi (F_0(\hat{v}))d\mathscr{L}^{\mathbf{n}}(\hat{v})
+\sup|\phi|\beta_0\beta^{{\mathbf{n}}}
\end{align}
for all bounded $\phi\in\mathcal{C}^{\infty}(\mathbb{R}^{\mathbf{n}+\mathbf{k}})$.
Here we used that by \eqref{flatbecomegraphpropb2} we have $JF_0\leq 1+C\beta_0$. 
Also we used that by \eqref{flatbecomegraphpropa1} the set
$M_{0}\cap\mathbf{B}(0,4)$ is contained in $\left(D_0\times\mathbf{B}^{\mathbf{k}}(0,1)\right)\cup S_0$.

Temporarily fix an arbitrary $y\in M_{t_0}\cap\mathbf{C}(0,2,2)$.
We want to use Theorem \ref{smoothregcor} 
with $x_0=y$, $r_0=\frac{\sqrt{t_0}}{\sqrt{8}}$, $R_0=1$ and $t_1=0$.
Thus for arbitrary $(s,x)\in [t_0-\frac{t_0}{2},t_0]\times\mathbf{B}(y,\frac{\sqrt{t_0}}{\sqrt{2}})$,
we have to show
\begin{align}
\label{flatbecomegraphprop31}
\int_{M_{0}\cap\mathbf{B}(0,4)}\Phi_{(s,x)}\leq 1 + d_0,
\end{align}
where $d_0$ is from Theorem \ref{smoothregcor}.
Here we used $C\sqrt{t_0}\leq C\beta_0\leq 1$ for $\beta_0$ small enough.
Also we used $\mathbf{B}(y,1)\subset\mathbf{B}(0,4)$.
By Definition \ref{heatkerneldef} 
we have
\begin{align*}
\int_{D_0}\Phi_{(s,x)}(0,F_0(\hat{v}))d\mathscr{L}^{\mathbf{n}}(\hat{v})
\leq
 \int_{\mathbf{B}^{\mathbf{n}}(0,4\rho_0)}\Phi_{(s,(\hat{x},0))}(0,(\hat{v},0))d\mathscr{L}^{\mathbf{n}}(\hat{v})
\leq 1.
\end{align*}
As well as $\sup|\Phi_{(s,x)}(0,\cdot)|
\leq\left(4\pi s\right)^{-\frac{\mathbf{n}}{2}}
\leq\left(2\pi t_0\right)^{-\frac{\mathbf{n}}{2}}
\leq\beta^{-\mathbf{n}}$.
Then for $\beta_0$ small enough,
inequality \eqref{flatbecomegraphprop23} with $\phi=\Phi_{(s,x)}(0,\cdot)$
establishes \eqref{flatbecomegraphprop31}.
Thus Theorem \ref{smoothregcor} yields
\begin{align}
\label{flatbecomegraphprop35}
\left|\mathbf{A}(M_{t_0},y)\right|\leq C_2t_0^{-\frac{1}{2}}
\end{align}
for all $y\in M_{t_0}\cap\mathbf{C}(0,2,2)$ for some constant $C_2\in (1,\infty)$.

In view of \eqref{flatbecomegraphprop11}
we can now apply Lemma \ref{multigraphlem}
with $r=1$, $K^2=C_2t_0^{-\frac{1}{2}}$ and $\xi^2=C_1t_0$ to obtain
an $m_0\in\mathbb{N}\cup\{0\}$
such that
\begin{align}
\label{flatbecomegraphprop43}
M_{t_0}\cap\mathbf{C}(0,1,1)=\bigcup_{i=1}^{m_0}\mathrm{graph}(g_i)
\end{align}
for functions $g_i:\mathbf{B}^{\mathbf{n}}(0,1)\to\mathbb{R}^{\mathbf{k}}$ 
that satisfy \eqref{flatbecomegraphpropd} for $t=t_0$.
Here we used $t_0\leq\beta_0^2$ and chose $\beta_0$ small enough.

We want to show $m_0\leq 1$.
Consider $\varphi$ from Definition \ref{barrierfctdef} 
and set $K:=\int\varphi^3(0,(\hat{x},0))d\mathscr{L}^{\mathbf{n}}(\hat{x})$.
We have
\begin{align*}
m_0K
&\leq
\int_{\mathbf{B}(0,1)}\varphi^3(0,(\hat{x},0))d\mu_{t_0}(x)
\leq
\int_{\mathbf{B}(0,1)}\varphi^3(t_0,x)d\mu_{t_0}(x) + C\beta_0
\\&\leq
\int_{\mathbf{B}(0,1)}\varphi^3(0,x)d\mu_{0}(x) + C\beta_0
\leq
(K+C\beta_0),
\end{align*}
which implies $m_0<2$ for $\beta_0$ small enough.
Here we used \eqref{flatbecomegraphprop43} for the first inequality,
\eqref{flatbecomegraphprop11} and $t_0\leq\beta_0^2$ for the second,
Corollary \ref{barrierlem} for the third
and \eqref{flatbecomegraphprop23} for the fourth.
Now \eqref{flatbecomegraphprop43} yields a
$g_{t_0}:\mathbf{B}^{\mathbf{n}}(0,1)\to\mathbb{R}^{\mathbf{k}}$
which satisfies \eqref{flatbecomegraphpropd} and such that
\begin{align}
\label{flatbecomegraphprop44}
\begin{split}
M_{t_0}\cap\mathbf{C}(0,1,1)
=\mathrm{graph}(g_{t_0})
\;\;\text{or}\;\;
M_{t_0}\cap\mathbf{C}(0,1,1)=\emptyset.
\end{split}
\end{align}
This statement holds for all $t_0\in J$.
Thus by continuity of $\mu_t\left(\mathbf{C}(0,1,1)\right)$ and \eqref{flatbecomegraphpropa2} we see, 
that always the first alternative has to be true,
then \eqref{flatbecomegraphprop44} implies the result.
\end{proof}

%
%
%
%
%
%
%
%
\begin{proof}[{Proof of Theorem \ref{flatstaygraphthm}}]
We may assume $a=0$ 
and $\rho=1$. In particular $f(0)=0$ and $\sup|f|\leq 2l$.
Fix $t\in I$.
Let $\hat{z}\in\mathbf{B}^{\mathbf{n}}(0,1)$ be arbitrary and set
\begin{align*}
\rho_0:=2^{-2}\sqrt{(l+t)^{-1}t},
\;\;\;
z:=(\hat{z},f(\hat{z})).
\end{align*}
Note that $4\rho_0+|\hat{z}|< 2$ 
and $4\rho_0+|\tilde{z}|< 2$.
In view of \eqref{flatstaygraphthma} and by $\mathrm{lip}(f)\leq l$ we see
\begin{align}
\label{flatstaygraphthm21}
M_{0}\cap\mathbf{C}(z,4\rho_0,2)
=M_{0}\cap\mathbf{C}(z,4\rho_0,4l\rho_0).
\end{align}
We want to use Proposition \ref{flatbecomegraphprop} 
with $z_0=z$, $s_1=0$, $s_2=\tau$, $\beta=2\sqrt{l}$ and $\varrho=\rho_0$.
Note that $C\sqrt{l_0}\leq\beta_0$
for $l_0$ small enough depending on $\beta_0$.
Then there exists a function
$g_{\hat{z}}:\mathbf{B}^{\mathbf{n}}(\hat{z},\rho_0)\to\mathbb{R}^{\mathbf{k}}$
with
\begin{align}
\label{flatstaygraphthm22}
\begin{split}
M_{t}\cap\mathbf{C}(z,\rho_0,\rho_0)
=\mathrm{graph}(g_{\hat{z}})
\;\;\text{or}\;\;
M_{t}\cap\mathbf{C}(z,\rho_0,\rho_0)=\emptyset.
\end{split}
\end{align}
Moreover we have
\begin{align}
\begin{split}
\label{flatstaygraphthm25}
\sup|g_{\hat{z}}|\leq \sup|f| + C t\leq 2l + Ct
\\
\sup|Dg_{\hat{z}}|\leq C\sqrt[4]{\rho_0^{-2}t}
\leq C\sqrt[4]{l+t},
\;\;\;
\sup|D^2g_{\hat{z}}|
\leq Ct^{-\frac{1}{2}}
\end{split}
\end{align}
where we used $\rho_0=2^{-2}\sqrt{(l+t)^{-1}t}$.
Here the height bound follows from Corollary \ref{heightboundcor}
combined with \eqref{flatstaygraphthma}, $\mathrm{lip}(f)\leq l$ and $0\in M_0$.

As $t\leq Cl_0\rho_0^2$ and $l_0$ can be chosen small,
we can use Corollary \ref{barrierlem} and \eqref{flatstaygraphthm21},
to see that \eqref{flatstaygraphthm22} actually holds in the larger cylinder
\begin{align*}
\begin{split}
M_{t}\cap\mathbf{C}((\hat{z},0),\rho_0,1)
=\mathrm{graph}(g_{\hat{z}})
\;\;\text{or}\;\;
M_{t}\cap\mathbf{C}((\hat{z},0),\rho_0,1)=\emptyset.
\end{split}
\end{align*}
Choosing different $\hat{z}\in\mathbf{B}^{\mathbf{n}}(0,1)$ ,
we obtain a graphical representation
$g_{t}:\mathbf{B}^{\mathbf{n}}(0,1)\to\mathbb{R}^{\mathbf{k}}$
which satisfies \eqref{flatstaygraphthmc} and such that
\begin{align}
\label{flatstaygraphthm24} 
\begin{split}
M_{t}\cap\mathbf{C}(0,1,1)
=\mathrm{graph}(g_{t})
\;\;\text{or}\;\;
M_{t}\cap\mathbf{C}(0,1,1)=\emptyset.
\end{split}
\end{align}
Statement \eqref{flatstaygraphthm24} holds for all $t\in I$.
Thus by continuity of $\mu_t\left(\mathbf{C}(0,1,1)\right)$ and \eqref{flatstaygraphthma} we see, 
that always the first alternative has to be true,
then \eqref{flatstaygraphthm24} implies the result.
\end{proof}

%
%
%
%
%
%
\begin{proof}[Proof of Remark \ref{flatbecomegraphrem}.2]
We may assume $s_1=0$ and $\varrho=1$.
Consider the setting of Proposition \ref{flatbecomegraphprop} 
with \eqref{flatbecomegraphpropa2} replaced by \eqref{flatbecomegraphpropb4}.
Then Theorem \ref{flatstaygraphthm} with
$\tau=s_2$, $a=(\hat{a}_0,f_0(\hat{a}_0))$, $l=l_0$ and $\rho=\sqrt{\beta}$
implies,
that \eqref{flatbecomegraphpropa2} holds nevertheless.
Here we used $\beta\leq\beta_0$ and chose $\beta_0$ small. 
\end{proof}

%
%
%
%
%
%
%
%
\subsection{Stay graphical for bounded curvature}
\label{smallcurve}
Consider an initial manifold that is graphical
with possibly large gradient, but bounded curvature.
Then we can use Theorem \ref{flatstaygraphthm} locally
to obtain the statement below.

%
%
%
%
\begin{prop}
\label{smallcurvestaygraphlem}
There exists a $C\in (1,\infty)$ and for all
$\kappa\in (0,1)$, $K\in [1,\infty)$ there exists a $\sigma_1\in (1,\infty)$
such that the following holds:
Let $\varrho,\Gamma,\tau\in (0,\infty)$, $t_0\in\mathbb{R}$, 
$z_0=(\hat{z}_0,\tilde{z}_0)\in\mathbb{R}^{\mathbf{n}+\mathbf{k}}$ 
and let $(M_t)_{t\in [t_0,t_0+\tau)}$ be a \mcf\; in $\mathbf{C}(0,2\varrho,\Gamma+\varrho)$.
Suppose there exists an $u:\mathbf{B}^{\mathbf{n}}(\hat{z}_0,\varrho)\to\mathbb{R}^{\mathbf{k}}$ 
with $\sup|u-\tilde{z}_0|<\Gamma-\frac{\varrho}{2}$ and
\begin{align}
\label{smallcurvestaygraphlema} 
M_{t_0}\cap\mathbf{C}(z_0,2\varrho,\Gamma+\varrho)
=\mathrm{graph}(u).
\end{align}
Moreover suppose
\begin{align}
\label{smallcurvestaygraphlemb1}
\|\mathbf{T}(M_{t_0},x)_{\natural}
-\left(\mathbb{R}^{\mathbf{n}}\times\{0\}^{\mathbf{k}}\right)_{\natural}\|
&\leq 1-2\kappa,
\\
\label{smallcurvestaygraphlemb2}
\left|\mathbf{A}(M_{t_0},x)\right|
&\leq K\varrho^{-1}
\end{align}
for all $x\in M_{t_0}\cap\mathbf{C}(z_0,2\varrho,\Gamma+\varrho)$.
Set $I:=(t_0,t_0+\sigma_1\varrho^2)\cap (t_0,t_0+\tau)$.
Then there exists a
$v:I\times\mathbf{B}^{\mathbf{n}}(\hat{z}_0,\varrho)\to\mathbb{R}^{\mathbf{k}}$
with
\begin{align}
\label{smallcurvestaygraphlemc} 
M_{t}\cap\mathbf{C}(z_0,\varrho,\Gamma)
=\mathrm{graph}(v(t,\cdot))
\end{align}
for all $t\in I$. Moreover we have
\begin{align}
\label{smallcurvestaygraphlemd1}
\|\mathbf{T}(M_t,x)_{\natural}
-\left(\mathbb{R}^{\mathbf{n}}\times\{0\}^{\mathbf{k}}\right)_{\natural}\|
&\leq 1-\kappa,
\\
\label{smallcurvestaygraphlemd2}
\left|\mathbf{A}(M_{t},x)\right|
&\leq C(t-t_0)^{-\frac{1}{2}}
\end{align}
for all $t\in I$ and for all $x\in M_t\cap\mathbf{C}(z_0,\varrho,\Gamma)$.
\end{prop}
%
%
\begin{rem}
See Chen and Yin \cite[7.5]{chenyin} for a better curvature estimate than \eqref{smallcurvestaygraphlemd2}.
\end{rem}
%
%
\begin{proof}
We may assume $z_0=0$, $t_0=0$ and $\varrho=1$.
Fix $s\in I$.
Let $x=(\hat{x},\tilde{x})\in M_s\cap \mathbf{C}(0,1,\Gamma)$ 
be arbitrary.
Consider $r_0$ and $z$ such that
\begin{align*}
r_0:=(c_1\kappa)^{5}K^{-1},
\;\;\;
z=(\hat{z},\tilde{z})\in M_0\cap\mathbf{B}(x,r_0)
\end{align*}
for some constant $c_1\in (0,1)$, which will be chosen below.
Note that such a $z$ always exists, by Corollary \ref{barrierlem} 
and $s\leq\sigma_1\leq cr_0^2$.
By \eqref{smallcurvestaygraphlema}
and for $c_1$ small enough we can estimate
\begin{align}
\label{smallcurvestaygraphlem12}
|\tilde{x}|
\leq|\tilde{z}|+r_0
\leq\sup|u|+\frac{1}{2}
<\Gamma.
\end{align}

Let $S\in\mathbf{SO}(\mathbf{n}+\mathbf{k})$ be such that 
$S[\mathbb{R}^{\mathbf{n}}\times\{0\}^{\mathbf{k}}]=\mathbf{T}(M_0,z)$.
In view of \eqref{smallcurvestaygraphlema} and \eqref{smallcurvestaygraphlemb2}
we can apply Corollary \ref{flatparacor}
with $R=K^{-1}(c_1\kappa)^4$, $a=z$,  $L=C\kappa^{-1}$ 
and $\alpha=(c_1\kappa)^4$ to obtain a 
$g_{0}:\mathbf{B}^{\mathbf{n}}(0,K^{-1}(c_1\kappa)^4)\to\mathbb{R}^{\mathbf{k}}$ 
with  
\begin{align}
\label{smallcurvestaygraphlem22}
M_0\cap\mathbf{B}(\hat{z},8r_0)
\subset S(\mathrm{graph}(g_{0}))+z
\subset M_0.
\end{align}
Here we estimated $(c_1\kappa)^4\leq\alpha_0$ 
and $8r_0\leq\left(C\kappa^{-1}\right)^{-1}K^{-1}(c_1\kappa)^4$, 
for $c_1$ small enough.
Also we used that by 
\eqref{tiltderivativerema} and \eqref{smallcurvestaygraphlemb1}
we have $\sup|Df_0|\leq C\kappa^{-1}$.
Corollary \ref{flatparacor} then implies the following bounds
\begin{align}
\label{smallcurvestaygraphlem23}
\sup|g_{0}|\leq CK^{-1}(c_1\kappa)^8\leq r_0.
\;\;\;
\sup|Dg_{0}|\leq (c_1\kappa)^4\leq l_0,
\end{align}
where $l_0$ is the constant from Theorem \ref{flatstaygraphthm}.
Here we used $r_0=(c_1\kappa)^{5}K^{-1}$, $\kappa\leq 1$
and we chose $c_1$ small enough.

For $t\in [0,\tau)$ set $N_t:=S^{-1}[M_t-z]$.
In view of \eqref{smallcurvestaygraphlem22} and \eqref{smallcurvestaygraphlem23} 
we can apply Theorem \ref{flatstaygraphthm}
with $(M_t)$ replaced by $(N_t)$, 
$a=0$, $\rho=4r_0$, $\Gamma=4r_0$ and $l=(c_1\kappa)^4$.
Thus we obtain a 
$g_{s}:\mathbf{B}^{\mathbf{n}}(0,2r_0)\to\mathbb{R}^{\mathbf{k}}$ 
with  
$N_s\cap\mathbf{C}(0,2r_0,2r_0)
=\mathrm{graph}(g_{s}))$.
In particular as $z\in M_0\cap\mathbf{B}(x,r_0)$,
there exists a $v:=(\hat{v},g_{s}(\hat{v}))\in\mathrm{graph}(g_{s}))$ 
such that $x=Sv+z$.
Also $g_{s}$ satisfies
\begin{align}
\label{smallcurvestaygraphlem32}
\sup|Dg_{s}|\leq C\sqrt[4]{(c_1\kappa)^4+s}\leq Cc_1\kappa,
\;\;\;
\sup|D^2g_{s}|\leq C s^{-\frac{1}{2}}.
\end{align}
Here we used $s\leq\sigma_1\leq cr_0^2$.
Using $S[\mathbb{R}^{\mathbf{n}}]=\mathbf{T}(M_0,z)$,
$S[\mathbf{T}(N_s,v)]=\mathbf{T}(M_s,x)$,
\eqref{smallcurvestaygraphlemb1}, \eqref{smallcurvestaygraphlem32} and \eqref{tiltderivativerema}
we have
\begin{align*}
\|\mathbf{T}(M_s,x)_{\natural}-(\mathbb{R}^{\mathbf{n}})_{\natural}\|
&\leq
\|\mathbf{T}(N_s,v)_{\natural}
-(\mathbb{R}^{\mathbf{n}})_{\natural}\|
+
\|(\mathbb{R}^{\mathbf{n}})_{\natural}
-\mathbf{T}(M_0,z)_{\natural}\|
\\&\leq C c_1\kappa+1-2\kappa
\leq 1-\kappa,
\end{align*}
where we identified $\mathbb{R}^{\mathbf{n}}$ with $\mathbb{R}^{\mathbf{n}}\times\{0\}^{\mathbf{k}}$ 
and we chose $c_1$ small enough.
Similarly using \eqref{smallcurvestaygraphlem32} and \eqref{secderivativerema} yields
\begin{align*}
\left|\mathbf{A}(M_{s},x)\right|
=\left|\mathbf{A}(N_{s},v)\right|
&\leq C s^{-\frac{1}{2}}.
\end{align*}
As $s$ and $x$ were arbitrary this already establishes 
\eqref{smallcurvestaygraphlemd1} and \eqref{smallcurvestaygraphlemd2}.
In view of \eqref{smallcurvestaygraphlem12} and \eqref{smallcurvestaygraphlemd1}
we can use Lemma \ref{constsheatlem} 
with $[t_1,t_2)=[0,\sigma_1)$, $r=1$ and $\Gamma_0=\Gamma$
to obtain the existence of the desired $v$.
Here we used that by \eqref{smallcurvestaygraphlema} the $m_0$
from Lemma \ref{constsheatlem} has to be $1$.
\end{proof}

%% file: one_codim.tex
Here we consider mean curvature flows of hypersurfaces.
In particular all results from the first part carry over with $\mathbf{k}=1$.
Having only one co-dimension allows the usage of the local estimates by Ecker and Huisken from \cite{eckerh1}.
We state the two theorems we need below.

%
%
%
%
%
%
%
%
%
%
\setcounter{thm}{0}

%
%
%
%
\begin{thm}[{\cite[2.1]{eckerh1}}]
\label{localgradthm}
Let $\varrho\in (0,\infty)$, $t_1\in\mathbb{R}$, $t_2\in (t_1,\infty)$,  $x_0\in\mathbb{R}^{\mathbf{n}+1}$
and let $(M_t)_{t\in [t_1,t_2)}$ 
be a \mcf\; in $\mathbf{B}(x_0,\varrho)$.
Set $\varrho(t):=\sqrt{(\varrho^2-2n(t-t_1))_+}$
and suppose
\begin{align*}
\nu_t(x)\cdot\mathbf{e}_{\mathbf{n}+1}>0
\end{align*}
for all $x\in M_t\cap\mathbf{B}(x_0,\varrho(t))$
for all $t\in[t_1,t_2)$.
Set $v(t,x):=(\nu_t(x)\cdot\mathbf{e}_{\mathbf{n}+1})^{-1}$.
Then
\begin{align*}
&v(t,x)
\left(1-\varrho^{-2}\left(|x-x_0|^2 + 2n(t-t_1)\right)\right)
\leq\sup_{\hat{x}\in\mathbf{B}^{\mathbf{n}}(\hat{x}_0,\varrho)}v(t_1,x)
\end{align*}
holds for all $x\in\mathbf{B}(x_0,\varrho(t))$
for all $t\in[t_1,t_2)$.
\end{thm}
%
%
%
%
\begin{thm}[{\cite[3.2(ii)]{eckerh1}}]
\label{localcurvethm}
There exists a $C\in (1,\infty)$ such that the following holds: 
Let $\varrho,\Gamma\in (0,\infty)$, $t_1\in\mathbb{R}$, $t_2\in (t_1,\infty]$,  $x_0\in\mathbb{R}^{\mathbf{n}+1}$
and let $(M_t)_{t\in [t_1,t_2)}$ 
be a \mcf\; in $\mathbf{C}(x_0,2\varrho,\Gamma)$.
Suppose there exists an
$f:(t_1,t_2)\times\mathbf{B}^{\mathbf{n}}(\hat{x}_0,2\varrho)\to\mathbb{R}$
such that
\begin{align*}
M_t\cap\mathbf{C}(x_0,2\varrho,\Gamma)
=
\mathrm{graph}(f(t,\cdot))
\end{align*}
for all $t\in (t_1,t_2)$.
Let $s_1\in (t_1,t_2)$.
Then
\begin{align*}
&\left|\mathbf{A}(M_s,(\hat{x},f(t,\hat{x}))\right|^2
\leq 
C\left((s-s_1)^{-1}+\varrho^{-2}\right)
\sup_{t\in [s_1,s]}\sup_{\mathbf{B}^{\mathbf{n}}(\hat{x}_0,2\varrho)}(1+|Df(t,\cdot)|^2)^2
\end{align*}
holds for all $x\in\mathbf{B}^{\mathbf{n}}(x_0,\varrho)$ for all $s\in (s_1,t_2)$.
\end{thm}

%
%
%
%
%
%
\subsection{Stay graphical for bounded gradient}
\label{stayinggraphical}
In this section we show Theorem \ref{stayinggraphicalthm},
which implies Theorem \ref{mainresult2}.
The main ingredient of the proof is Lemma \ref{maintaingraphlem},
which originally appears in the author's thesis \cite[12.11]{lahiri}.
This Lemma combines the curvature bound from Ecker and Huisken \cite[3.2(ii)]{eckerh1}
with Proposition \ref{smallcurvestaygraphlem}, 
to maintain the graphical representation of a \mcf\;
that has been graphical over a period of time.

%
%
%
%
\begin{thm}
\label{stayinggraphicalthm}
For every $L\in [1,\infty)$ there exists a $\Lambda\in (1,\infty)$ such that the following holds:
Let $\rho,\Gamma,\tau\in (0,\infty)$, $\delta\in (0,1]$,
$t_0\in\mathbb{R}$, $\hat{a}\in\mathbb{R}^{\mathbf{n}}$, $a:=(\hat{a},0)$
and let $(M_t)_{t\in [t_0,t_0+\tau)}$ be a \mcf\; in $\mathbf{C}(a,\rho,\Gamma)$.
Suppose there exists an
$f:\mathbf{B}^{\mathbf{n}}(\hat{a},\rho)\to\mathbb{R}$
with $\sup|f|\leq\Gamma-2\delta\rho$, $\sup |Df|\leq L$
and
\begin{align}
\label{stayinggraphicalthma}
M_{t_0}\cap\mathbf{C}(a,\rho,\Gamma)
= 
\mathrm{graph}(f).
\end{align}
Set $\sigma(t):=\Lambda\sqrt{t-t_0}$.
Then for all $t\in (t_0,t_0+\tau)$ such that $\sigma(t)<\delta\rho$
there exists an 
$g_t:\mathbf{B}^{\mathbf{n}}(\hat{a},\rho-\sigma(t))\to\mathbb{R}$
with $\sup|g_t|\leq\sup|f| + \sigma(t)$, $\sup |Dg_t|\leq 4L$ and
\begin{align}
\label{stayinggraphicalthmb}
M_{t_0}\cap\mathbf{C}(a,\rho-\sigma(t),\Gamma-\sigma(t))
= 
\mathrm{graph}(g_t).
\end{align}
\end{thm}
%
%
\begin{rem}
\begin{enumerate}
\item
The curvature estimate by Ecker and Huisken 
\cite[3.2(ii)]{eckerh1}(see Theorem \ref{localcurvethm})
can be used to obtain bounds on the curvature and $|D^2g_t|$ (via \eqref{secderivativerema}).
\item
The bounds on $|g_t|$, $|Dg_t|$
may be improved if you are further inside the graphical cylinder.
Use Corollary \ref{heightboundcor}
or the gradient estimate by Ecker and Huisken \cite[2.1]{eckerh1}(see Theorem \ref{localgradthm}) respectively.
\item
For $a=0$, $t_0=0$, $\rho=\Gamma=2$ and $\delta=\frac{1}{4}$
this implies Theorem \ref{mainresult2}.
Here we use that for $t\in (0,\kappa_L]$ we have $\sigma(t)<\frac{1}{2}$,
for $\kappa_L<(2\Lambda)^{-2}$.
\item
See also Example \ref{shrinkingsquareexmp},
which shows,
that it is reasonably that the set where graphical representation is maintained shrinks in time.
\end{enumerate}
\end{rem}

%
%
%
%
The proof of Theorem \ref{stayinggraphicalthm} 
is based on the following Lemma
which is taken from \cite[12.11]{lahiri}.
Here we give a much shorter proof.
%
%
\begin{lem}[{\cite[12.11]{lahiri}}]
\label{maintaingraphlem}
For every $L_1\in [1,\infty)$ 
there exists a $\lambda_1\in (0,1)$ 
such that the following holds:
Let $\rho_1,R_1,\Gamma_1\in (0,\infty)$, 
$s_0\in\mathbb{R}$, 
and let $(M_t)_{t\in [s_0-R_1^2,s_0+\lambda_1^2R_1^2)}$ 
be a \mcf\; in $\mathbf{C}(0,\rho_1+R_1,\Gamma_1+R_1) $.
Set $J:=(s_0-R_1^2,s_0)$.
Suppose there exists an
$u:J\times\mathbf{B}^{\mathbf{n}}(0,\rho_1+R_1)\to\mathbb{R}$
with $\sup|u|<\Gamma_1 - R_1$, $\sup |Du|\leq L_1$
and
\begin{align}
\label{maintaingraphlema}
M_t\cap\mathbf{C}(0,\rho_1+R_1,\Gamma_1+R_1) 
=\mathrm{graph}(u(t,\cdot))
\end{align}
for all $t\in J$.
Set $I:=(s_0,s_0+\lambda_1^2R_1^2)$.
Then there exists a function
$v:I\times\mathbf{B}^{\mathbf{n}}(0,R_1)\to\mathbb{R}$
with $\sup|v|\leq\sup|u| + R_1$, $\sup |Dv|\leq 4L_1$ and 
\begin{align}
\label{maintaingraphlemb} 
M_t\cap\mathbf{C}(0,\rho_1,\Gamma_1) 
=\mathrm{graph}(v(t,\cdot))
\end{align}
for all $t\in I$.
\end{lem}
%
%
\begin{proof}
We may assume $s_0=0$ and $R_1=4$.
Let $z\in\mathbf{B}^{\mathbf{n}}(0,\rho_1)\times\{0\}$ be arbitrary.
Note that $1+\sup|Du|^2\leq 2L_1^2$,
so \eqref{hypertilt} yields
\begin{align}
\label{maintaingraphlem11} 
\|\mathbf{T}(M_{t_0},x)_{\natural}
-\left(\mathbb{R}^{\mathbf{n}}\times\{0\}^{\mathbf{k}}\right)_{\natural}\|
&= 1-\frac{1}{1+|Du(t_0,x)|^2}
\leq 1-2(2L_1)^{-2}
\end{align}
for all $x\in M_t\cap\mathbf{C}(z,4,\Gamma_1+4)$
for all $t\in\left[-16,0\right]$.

Also by Theorem \ref{localcurvethm} 
with $t_2=0$, $t_1=-16$, $R=2$ and $\Gamma=\Gamma_1+4$
we obtain
\begin{align}
\label{maintaingraphlem21} 
&\left|\mathbf{A}(M_t,(\hat{x},f(t,\hat{x}))\right|
\leq 
CL_1^2=:K
\end{align}
for all $\hat{x}\in\mathbf{B}^{\mathbf{n}}(\hat{z},2)$ 
and all $t\in (-1,0)$.

In view of \eqref{maintaingraphlem11}, \eqref{maintaingraphlem21}
and \eqref{maintaingraphlema} we can apply Proposition \ref{smallcurvestaygraphlem}
with $t_0=-2^{-1}\sigma_0$, $\varrho=1$, $\Gamma=\Gamma_1$ and $\kappa=(2L_1)^{-2}$.
As $z\in\mathbf{B}^{\mathbf{n}}(0,\rho_1)\times\{0\}$
was arbitrary, this establishes the result.
Here we chose $\lambda_1^2\leq 2^{-1}\sigma_0$.
\end{proof}

%
%
\begin{proof}[Proof of Theorem \ref{stayinggraphicalthm}]
We may assume $t_0=0$ and $a=0$.
Let $\epsilon\in (0,2^{-2}\delta]$.
For $t\in [0,\tau)$ set
\begin{align*}
\rho_{\epsilon}(t)&:=\rho-\epsilon\rho-2^{-2}\Lambda\sqrt{t},
\\
\Gamma_{\epsilon}(t)&:=\Gamma-\epsilon\rho-2^{-2}\Lambda\sqrt{t},
\\
\gamma_{\epsilon}(t)&:=\sup|f|+\epsilon\rho+2^{-2}\Lambda\sqrt{t}.
\end{align*}
A time $t\in [0,\tau)$ is called proper,
if $\sigma(t)\geq\delta\rho$, or there exists a function
$f_{\epsilon,t}:\mathbf{B}^{\mathbf{n}}(0,\rho_{\epsilon}(t))\to\mathbb{R}$
such that
\begin{align}
\label{stayinggraphicalthm21}
M_t\cap\mathbf{C}(0,\rho_{\epsilon}(t),\Gamma_{\epsilon}(t))
=\mathrm{graph}(f_{\epsilon,t}),
\\
\label{stayinggraphicalthm22}
\sup|f_{\epsilon,t}|\leq\gamma_{\epsilon}(t),
\;\;\;
\sup |Df_{\epsilon,t}|\leq 4L.
\end{align}
Note that if $\sigma(t)<\delta\rho$, we have $\rho_{\epsilon}(t)>0$.
We consider the set
\begin{align*}
I_{\epsilon}:=\left\{s\in (0,\tau]:
\;t\;\text{is proper for all}\; t\in [0,s)\right\}.
\end{align*}
For $s\in I_{\epsilon}$ with $\sigma(s)<\delta\rho$
we have $\sqrt{s}<\Lambda^{-1}\delta\rho$.
Thus as $\epsilon\leq\frac{\delta}{4}$ and $\sup|f|\leq\Gamma-2\delta\rho$
we can estimate
\begin{align}
\label{stayinggraphicalthm32}
\gamma_{\epsilon}(s)
< \sup|f|+\frac{\delta\rho}{2}
\leq\Gamma-\frac{3\delta\rho}{2}
\leq\Gamma_{\epsilon}(s)-\Lambda\sqrt{s}
\end{align}
for all $s\in I_{\epsilon}$ with $\sigma(s)<\delta\rho$.

By continuity in time and 
$\mathbf{C}(0,\rho-\epsilon,\Gamma-\epsilon)\subset\subset\mathbf{C}(0,\rho,\Gamma)$ 
we have $I_{\epsilon}\neq\emptyset$.
Consider $s\in I_{\epsilon}$ with $s<\tau$.
We want to show, that there exists an $s_2\in (s,\tau]$ such that $(0,s_2]\subset I_{\epsilon}$.
If $\sigma(s)\geq\delta\rho$ we directly see $(0,\tau]=I_{\epsilon}$.
Thus assume $\sigma(s)< \delta\rho$.
Let $\lambda_1\in (0,1)$ be from Lemma \ref{maintaingraphlem},
chosen for $L_1=4L$.
Set
\begin{align*}
s_2:=\min\left\{(1+2^{-3}\lambda_1^2)s,\tau\right\},
\;\;\;
s_0:=s_2-2^{-2}\lambda_1^2s,
\end{align*}
in particular
$2^{-1}s<s_0<s<s_2<2s_0$
and $\sqrt{s}-\sqrt{s_0}\geq c\lambda_1^2\sqrt{s}$.
By definition of $\rho_{\epsilon}$, $\Gamma_{\epsilon}$ and $\gamma_{\epsilon}$ 
this yields
\begin{align}
\label{stayinggraphicalthm51}
\rho_{\epsilon}(s_0)-\sqrt{s}
&\geq \rho_{\epsilon}(s)+4\sqrt{\mathbf{n}s}
=:\rho_1,
\\
\label{stayinggraphicalthm52}
\Gamma_{\epsilon}(s_0)
&\geq \Gamma_{\epsilon}(s)+\sqrt{s},
\\
\label{stayinggraphicalthm53}
\gamma_{\epsilon}(s_0)
&\leq \gamma_{\epsilon}(s)-\sqrt{s},
\end{align}
where we chose $\Lambda$ large depending on $\lambda_1$.

By \eqref{stayinggraphicalthm22} we have $\sup|Df_{\epsilon,t}|\leq 4L$ for all $t\in[0,s_0]$.
Also, by \eqref{stayinggraphicalthm22}, \eqref{stayinggraphicalthm32} and \eqref{stayinggraphicalthm53}
we obtain
\begin{align}
\label{stayinggraphicalthm63}
\sup|f_{\epsilon,t}(\hat{x})|
\leq\sup_{t\in[0,s_0]}\gamma_{\epsilon}(t)
=\gamma_{\epsilon}(s_0)
\leq\gamma_{\epsilon}(s)-\sqrt{s}
<\Gamma_{\epsilon}(s)-\sqrt{s}.
\end{align}
Combining \eqref{stayinggraphicalthm21}
with \eqref{stayinggraphicalthm51}, \eqref{stayinggraphicalthm52} and \eqref{stayinggraphicalthm63} we have
\begin{align*}
M_t\cap\mathbf{C}(0,\rho_1+\sqrt{s},\Gamma_{\epsilon}(s)+\sqrt{s})
=\mathrm{graph}(f_{\epsilon,t})
\end{align*}
for all $t\in (0,s_0]$, where we consider the restrictions of $f_{\epsilon,t}$ 
to $\mathbf{B}^{\mathbf{n}}(\hat{y},\rho_1+\sqrt{s})$.
Then by Lemma \ref{maintaingraphlem} with $L_1=4L$, $R_1=2^{-1}\sqrt{s}$
and  $\Gamma_1=\Gamma_{\epsilon}(s)$
there exists a 
$g_{\epsilon}:(0,s_2)\times\mathbf{B}^{\mathbf{n}}(0,\rho_1)\to\mathbb{R}$
with
\begin{align}
\label{stayinggraphicalthm61}
M_t\cap\mathbf{C}(0,\rho_1,\Gamma_{\epsilon}(s))
=\mathrm{graph}(g_{\epsilon}(t,\cdot))
\end{align}
for all $t\in (0,s_2)$.
Moreover 
$\sup|g_{\epsilon}|\leq\gamma_{\epsilon}(s)$.

It remains to show that $\sup|Dg_{\epsilon}|\leq 4L$.
Let $t_2\in (0,s_2)$ 
and $\hat{x}_0\in\mathbf{B}^{\mathbf{n}}(0,\rho_{\epsilon}(s))$ 
be arbitrary.
Set $x_0=(\hat{x}_0,(g_{\epsilon}(t_2,\hat{x}_0))$ 
and $\varrho:=\rho_1-\rho_{\epsilon}(s)=4\sqrt{\mathbf{n}s}$.
Then by $\sup|g_{\epsilon}|\leq\gamma_{\epsilon}(s)$, \eqref{stayinggraphicalthm32}
and for $\Lambda$ large enough we have
\begin{align*}
\mathbf{B}(x_0,\varrho)\subset\mathbf{C}(0,\rho_1,\Gamma_{\epsilon}(s)).
\end{align*}
In view of \eqref{stayinggraphicalthm61}, \eqref{stayinggraphicalthma}, 
\eqref{hypergraphnormal} and $\sqrt{1+\sup|Df|^2}\leq 2L$
we can now use Theorem \ref{localgradthm} to obtain $|Dg_{\epsilon}(t_2,\hat{x}_0)|\leq 4L$.
Thus we have $s_2\in I_{\epsilon}$.
Note that $s_2$ is either $\tau$ or $(1+2^{-3}\lambda_1^2)s$,
so the amount by which we enlarge the time interval is actually increasing with $s$.
As $s\in I_{\epsilon}$ was arbitrary and $I_{\epsilon}\neq\emptyset$
this yields $\tau\in I_{\epsilon}$.
Note that
\begin{align*}
\rho_{\epsilon}(t)\geq\rho-\sigma(t)-\epsilon\rho,
\;\;\;
\Gamma_{\epsilon}(t)\geq\Gamma-\sigma(t)-\epsilon\rho,
\;\;\;
\gamma_{\epsilon}(t)\leq \sup|f|+\sigma(t)+\epsilon\rho.
\end{align*}
As $\epsilon$ can be chosen arbitrary small we established the result.
\end{proof}

%
%
%
%
Additionally assuming a height bound in Theorem \ref{stayinggraphicalthm}, 
yields that the graphical representation will have small gradient after some time.
%
%
\begin{lem}
\label{flatgraphlem}
For every $L_0\in [1,\infty)$ there exist 
$\Lambda_1\in (1,\infty)$ and $\gamma_1\in (0,1)$ 
such that the following holds:
Let $\gamma\in (0,\gamma_1]$, $\rho_0\in (0,\infty)$,
$t_1\in\mathbb{R}$, 
$\hat{z}_0\in\mathbb{R}^{\mathbf{n}}$,
$z_0:=(\hat{z}_0,0)$,
$s_1:=t_1+\gamma^4\rho_0^2$, $t_2\in (s_1,\infty)$
and let $(M_t)_{t\in [t_1,t_2)}$ be a \mcf\; in $\mathbf{C}(z_0,4\rho_0,4\rho_0)$.
Suppose there exists an
$u:\mathbf{B}^{\mathbf{n}}(\hat{z}_0,4\rho_0)\to\mathbb{R}$ 
with $\sup |Du|\leq L_0$, $\sup|u|\leq\gamma^4\rho_0$
and
\begin{align}
\label{flatgraphlema} 
M_{t_1}\cap\mathbf{C}(z_0,4\rho_0,4\rho_0)
=\mathrm{graph}(u).
\end{align}
Then there exists a 
$v:\mathbf{B}^{\mathbf{n}}(\hat{z}_0,\rho_0)\to\mathbb{R}$
with
\begin{align}
\label{flatgraphlemb} 
M_{s_1}\cap\mathbf{C}(z_0,\rho_0,\rho_0)
=\mathrm{graph}(v)
\end{align}
and
\begin{align}
\label{flatgraphlemc}
\sup|v|\leq \Lambda_1\gamma^4\rho_0,
\;\;\;
\sup|Dv|\leq \Lambda_1\gamma,
\;\;\;
\sup|D^2v|\leq \Lambda_1\gamma^{-2}\rho_0^{-1}.
\end{align}
\end{lem}
%
%
%
%
%
\begin{proof}
We may assume $t_1=0$, $z_0=0$ and $\rho_0=1$.
Let $\Lambda$ be from Theorem \ref{stayinggraphicalthm}, chosen for $L=L_0$.
Then we can estimate
$\Lambda\sqrt{2s_1}\leq\frac{1}{4}$              
and $\sup|u|\leq\frac{1}{2}$,
where we used $\gamma\leq\gamma_1$
and chose $\gamma_1$ small depending on $\Lambda$.
Hence we can use Theorem \ref{stayinggraphicalthm} 
with $\rho=\Gamma=4$, $a=0$, $t_0=0$, $\tau=t_2$, $L=L_0$ 
and $\delta=\frac{1}{4}$.
This yields a function 
$g:[0,2s_1]\times\mathbf{B}^{\mathbf{n}}(0,3)\to\mathbb{R}$
with $\sup|g(\hat{x})|<\rho_0$,
\begin{align}
\label{flatgraphlem11} 
M_{t}\cap\mathbf{C}(0,3,3) 
=\mathrm{graph}(g(t,\cdot))
\end{align}
for all $t\in [0,2s_1]$.
Now we can apply Theorem \ref{localcurvethm}
with $\varrho=1$, $\Gamma=3$ and arbitrary $x_0\in \mathbf{B}^{\mathbf{n}}(0,2)\times\{0\}$,
to obtain
\begin{align}
\label{flatgraphlem12} 
\left|\mathbf{A}(M_{s_1},x)\right|
\leq CL_0^4 s_1^{-\frac{1}{2}}
=CL_0^4\gamma^{-2}
\end{align}
for all $x\in M_{s_1}\cap\mathbf{C}(0,2,3)$.

As $\sup|u|\leq\gamma^4$, Corollary \ref{heightboundcor} 
with $r_0=\gamma^4$, $R=2$ and $x_0=0$ implies
\begin{align}
\label{flatgraphlem21}
M_{t}\cap\mathbf{C}(0,2,2)
\subset\mathbf{C}(0,2,C\gamma^4)
\end{align}
for all $t\in [t_1,s_1]$.

Let $v$ be the restriction of $g(s_1,\cdot)$ to $\mathbf{B}^{\mathbf{n}}(0,1)$.
In particular \eqref{flatgraphlem11} implies \eqref{flatgraphlemb}.
In view \eqref{flatgraphlem12} and \eqref{flatgraphlem21}   
we can apply Lemma \ref{multigraphlem} 
with $r=1$, $\xi=C\gamma^2$ and $K=CL_0^2\gamma^{-1}$
to see that \eqref{flatgraphlemc} holds.
Here we estimated $K\xi\leq CL_0^2\gamma\leq\alpha_1$.
Note that by \eqref{flatgraphlem11} we have that $m_0=1$,
so the function from \eqref{multigraphlemd} has to coincide with $v$.
\end{proof}

In the following example we construct a \mcf\; (actually a curve shortening flow) 
that initially intersects the cylinder $(-2,2)\times (-2,2)$ in the straight line $(-2,2)\times\{0\}$,
but becomes non-graphical inside that cylinder immediately
and also becomes non-graphical inside the cylinder of half the radius after finite time.

\begin{exmp}
\label{shrinkingsquareexmp}
Consider $\epsilon\in (0,4^{-1})$, $\mathbf{S}^1:=\partial\mathbf{B}^2(0,1)$
and $A\subset\mathbb{R}^2$ with
\begin{align*}
[-2,2]\times[0,2]\subset A\subset [-(2+\epsilon),2+\epsilon]\times [0,2+\epsilon]
\end{align*}
and such that $\partial A=\Psi_0[\mathbf{S}^1]$ for a smooth embedding $\Psi_0$.
In particular $M_0:=\partial A$ is graphical inside $(-2,2)\times (-2,2)$
with gradient zero.
By the work of Grayson \cite{grayson}
there exists a unique mean curvature flow $(M_t)_{t\in [0,T)}$
starting from $M_0$ and shrinking to a round point for $t\nearrow T$.
By choice of $A$ 
and Brakke`s sphere comparison results \cite[3.7, 3.9]{brakke} (see also \cite[3.3]{eckerb})
we have
\begin{align}
\label{shrinkingsquareexmp11}
\mathbf{B}^2((0,1),r(t))\cap M_t=\emptyset,
\;\;\; 
M_t\subset \mathbf{B}^2((0,1),R(t))
\end{align}
for all $t\in [0,T)$, where $R(t):=\sqrt{3+3\epsilon-2t}$ and $r(t):=\sqrt{1-2t}$.

In particular $\frac{1}{2}\leq T\leq\frac{3+3\epsilon}{2}\leq 2$.
Statement \eqref{shrinkingsquareexmp11} also yields that $M_{2\epsilon}\subset\mathbf{B}^2((0,1),\sqrt{3-\epsilon})$. 
Then $M_{2\epsilon}$ cannot be graphical inside $(-2,2)\times (-2,2)$.
Moreover as $(M_t)$ shrinks to a round point,
there exists an $s\in (0,T)$ such that $M_s\subset\mathbf{B}^2(a,\epsilon)$
for some $a\in\mathbb{R}^2$.
Then $s< 2$ and  for any $h\in (0,2]$
the manifold $M_{s}$ cannot be graphical inside $(-2\epsilon,2\epsilon)\times (-h,h)$.
\end{exmp}

%
%
%
%
%
%
\subsection{Become graphical}
\label{becomegraph}
In this section we show Theorem \ref{becominggraphthm},
which implies Theorem \ref{mainresult4}.
The idea is that for an initially almost graphical \mcf\; that lies in a slab,
the gradient on the shrinking graphical part is decreasing (by Lemma \ref{flatgraphlem})
and the measure of the growing non-graphical part can be controlled.
It turns out that there is a time when both the gradient and this measure 
are small enough to apply Proposition \ref{flatbecomegraphprop},
which yields a graphical representation.
%
%
%
%
\begin{thm}
\label{becominggraphthm}
There exist constants $C\in (1,\infty)$ and $\sigma_0\in (0,1)$ and
for every $L\in [1,\infty)$ 
there exist $\Lambda_0\in (1,\infty)$ and $\lambda_0\in (0,1)$ 
such that the following holds:
Let $\lambda\in (0,\lambda_0]$, $\rho\in (0,\infty)$, $h,t_1\in\mathbb{R}$, 
$t_0:=t_1+\Lambda_0\lambda^2\rho^2$, $t_2\in (t_0,\infty)$,
$a=(\hat{a},\tilde{a})\in\mathbb{R}^{\mathbf{n}+1}$
and let $(M_t)_{t\in [t_1,t_2)}$ be a \mcf\; in $\mathbf{C}(a,2\rho,2\rho)$.
Assume
\begin{align}
\label{becominggraphthma} 
M_{t_1}\cap\mathbf{C}(a,2\rho,2\rho)
\subset\mathbf{C}(a,2\rho,\lambda^{2\mathbf{n}}\rho).
\end{align}
Suppose there exist an open subset $D^{\mathbf{n}}\subset\mathbf{B}^{\mathbf{n}}(\hat{a},2\rho)$
and an
$f:D^{\mathbf{n}}\to\mathbb{R}$
with $\sup |Df|\leq L$
such that
\begin{align}
\label{becominggraphthmb} 
M_{t_1}\cap\left(D^{\mathbf{n}}\times\mathbf{B}^{1}(\tilde{a},\rho)\right)
=\mathrm{graph}(f),
\end{align}
and for $E^{\mathbf{n}}:=\mathbf{B}^{\mathbf{n}}(\hat{a},2\rho)\setminus D^{\mathbf{n}}$
\begin{align}
\label{becominggraphthmc1} 
E^{\mathbf{n}}\subset\left(\mathbb{R}^{\mathbf{n}-1}\times\hball{h}{\lambda^{\mathbf{n}}\rho}\right),
\\
\label{becominggraphthmc2} 
\mathscr{H}^{\mathbf{n}}\left(M_{t_1}\cap\left(E^{\mathbf{n}}\times\mathbf{B}^{1}(\tilde{a},\rho)\right)\right)
\leq\lambda^{\mathbf{n}}\rho^{\mathbf{n}}.
\end{align}
Set $I:=[t_0,t_0+\sigma_0\rho^2)\cap [t_0,t_2)$.
Then there exists a 
$g:I\times\mathbf{B}^{\mathbf{n}}(\hat{a},\rho)\to\mathbb{R}$
with
\begin{align}
\label{becominggraphthmd} 
M_{t}\cap\mathbf{C}(a,\rho,\rho)
=\mathrm{graph}(g(t,\cdot))
\end{align}
and
\begin{align}
\begin{split}
\label{becominggraphthme}
\sup|g(t,\cdot)-\tilde{a}|&\leq C\rho_0^{-1}(t-t_1),
\\
\sup|Dg(t,\cdot)|&\leq C\sqrt[4]{\rho_0^{-2}(t-t_1)},
\\
\sup|D^2g(t,\cdot)|&\leq C(t-t_1)^{-\frac{1}{2}}
\end{split}
\end{align}
for all $t\in I$.
\end{thm}
%
%
\begin{rem}
Using this Theorem with $a=0$, $\rho=1$, $h=0$, $t_1=0$, $t_2=\tau$ and $\lambda^2=\Lambda_0^{-1}\epsilon$
implies Theorem \ref{mainresult4} with $\gamma_2=(\Lambda_0^{-1}\epsilon)^{\mathbf{n}}$ and $\kappa_2:=c\sigma_0$.
\end{rem}
%
%
\begin{proof}
We may assume $a=0$, $t_1=0$ and $\rho=4$.
Set
\begin{align*}
\rho_1&:=2^{-4}\lambda^{\mathbf{n}},
\;\;\;
B^1:=\mathbf{B}^1(0,1)=(-1,1),
\\
D_1^{\mathbf{n}}&:=\mathbf{B}^{\mathbf{n}}(0,5)
\setminus (\mathbb{R}^{\mathbf{n}-1}\times\overline{\mathbf{B}^1(h,8\lambda^{\mathbf{n}})})
\\
E_1^{\mathbf{n}}&:=\mathbf{B}^{\mathbf{n}}(0,5)
\cap (\mathbb{R}^{\mathbf{n}-1}\times\overline{\mathbf{B}^1(h,8\lambda^{\mathbf{n}})}).
\end{align*}
Consider $z=(\hat{z},0)\in D_1^{\mathbf{n}}\times\{0\}$.
Then by \eqref{becominggraphthmc1} and $\lambda\leq\lambda_0\leq 1$
we have
\begin{align*} 
\mathbf{C}(z,4\rho_1,4\rho_1))\subset D^{\mathbf{n}}\times B^1.
\end{align*}
In view of \eqref{becominggraphthma} and \eqref{becominggraphthmb}
we can use Lemma \ref{flatgraphlem}
with $L_0=L$, $z_0=z$, $\rho_0=\rho_1$, $\gamma=4\lambda^{\frac{\mathbf{n}}{4}}$.
Here we have to choose $\lambda_0$ small enough such that $4\lambda^{\frac{\mathbf{n}}{4}}\leq\gamma_0$,
which depends on $L$.
Set
\begin{align*}
s_1:=\gamma^4\rho_1^2=\lambda^{3\mathbf{n}}.
\end{align*}
Lemma \ref{flatgraphlem} yields a 
$v_{z}:\mathbf{B}^{\mathbf{n}}(\hat{z},\rho_1)\to\mathbb{R}$
with $\sup|Dv_{z}|\leq 4\Lambda_1\lambda^{\frac{\mathbf{n}}{4}}$ and
\begin{align*}
M_{s_1}\cap\mathbf{C}(z,\rho_1,\rho_1)
=\mathrm{graph}(v_{z}).
\end{align*}
Thus, as $\hat{z}\in D_1^{\mathbf{n}}$ was arbitrary,
we can combine the $v_{z}$ to obtain a
$v_2:D_1^{\mathbf{n}}\to\mathbb{R}$
with 
\begin{align}
\label{becominggraphthm31} 
M_{s_1}\cap\left(D_1^{\mathbf{n}}\times\chball{\rho_1}\right)
=\mathrm{graph}(v_2)
\end{align}
and
\begin{align}
\label{becominggraphthm32} 
\sup|Dv_2|\leq 4\Lambda_1\lambda^{\frac{\mathbf{n}}{4}}.
\end{align}

By \eqref{becominggraphthma} we can use Corollary \ref{heightboundcor} 
with $R=1$ and $x_0\in\mathbf{B}^{\mathbf{n}}(0,6)\times\{0\}$
as well as $R=4$ and $x_0=0$,
to obtain
\begin{align}
\label{becominggraphthm21}
M_{t}\cap\mathbf{C}(0,6,1)
&\subset\mathbf{C}(0,6,r(t))
\\
\label{becominggraphthm22}
M_{t}\cap\mathbf{C}(0,4,4)
&\subset\mathbf{C}(0,4,r(t))
\end{align}
for $r(t):=\lambda^{2\mathbf{n}}+Ct$ for all $t\in [0,t_2)$.
Note that $r(s_1)=C\lambda^{2\mathbf{n}}\leq 2^{-4}\lambda^{\mathbf{n}}=\rho_1$,
where we used $\lambda\leq\lambda_0$ and $\lambda_0$ small.
Hence with \eqref{becominggraphthm31} we can conclude
\begin{align}
\label{becominggraphthm33} 
M_{s_1}\cap\left(D_1^{\mathbf{n}}\times B^1\right)
=\mathrm{graph}(g_2).
\end{align}

We want to use Proposition \ref{flatbecomegraphprop}.
In order to do so, the $M_{s_1}$-measure in $E_1^{\mathbf{n}}\times B^1$ has to be small.
The idea is, that by Proposition \ref{localisingprop}
the $M_{s_1}$-measure in $E_1^{\mathbf{n}}\times B^1$
is bounded by the $M_{0}$-measure in some larger set $Y\times B^1$.
This set can then be estimated with \eqref{becominggraphthmb} and \eqref{becominggraphthmc2}.

By definition of $E_1^{\mathbf{n}}$ and the slab statement \eqref{becominggraphthm21}
we have
\begin{align}
\label{becominggraphthm54} 
M_{s_1}\cap(E_1^{\mathbf{n}}\times B^1)
\subset\mathbf{B}(0,6)
\cap\left(\mathbb{R}^{\mathbf{n}-1}\times\mathbf{B}^{2}\left((h,0),C_1\lambda^{\mathbf{n}}\right)\right)
\end{align}
for some constant $C_1\in (1,\infty)$.
Consider
\begin{align*}
Y^{\mathbf{n}}&:=\mathbf{B}^{\mathbf{n}}(0,8)\cap\left(\mathbb{R}^{\mathbf{n}-1}\times\mathbf{B}^1(h,2C_1\lambda^{\mathbf{n}})\right).
\end{align*}
We want to show
\begin{align}
\label{becominggraphthm51} 
\mathscr{H}^{\mathbf{n}}(M_{s_1}\cap\left(E_1^{\mathbf{n}}\times B^1\right))
\leq
C\mathscr{H}^{\mathbf{n}}\left(M_{0}\cap\left(Y^{\mathbf{n}}\times B^1\right)\right).
\end{align}
For $(u,v)\in\mathbb{R}^{\mathbf{n}-1}\times\mathbb{R}^2$ set $Q(u,v):=v$.
Consider $\eta,\zeta\in\mathcal{C}^{0,1}\left(\mathbb{R}\times\mathbf{B}(0,8)\right)$
given by
\begin{align*}
\eta(t,x)&:=\left(1-(2C_1\lambda^{\mathbf{n}})^{-2}(|Q(x)|^2+2\mathbf{n}t)\right)_+
\\
\zeta(t,x)&=\zeta(x):=\left(8^2-|x|^2\right)_+.
\end{align*}
Note that $\eta\in\mathcal{C}^2\left(\{\eta>0\}\right)$
and $\{\zeta(\cdot)\eta(0,\cdot)>0\}\subset Y\times B^1$.
Also $(\partial_t-\mathrm{div}_{M_t}D)\eta\leq 0$ on $\{\eta>0\}$, as $\mathrm{div}_{M_t}DQ\leq\mathbf{n}$.
Then Proposition \ref{localisingprop} with $y_0=0$, $L=\lambda^{-2\mathbf{n}}$ and $\varrho=8$
yields
\begin{align}
\label{becominggraphthm52}
\int_{\mathbb{R}^{\mathbf{n}+1}}
\left(\zeta(x)\eta(s_1,x)-2\mathbf{n}\lambda^{-2\mathbf{n}}s_1\right)^3
d\mu_{s_1}
\leq 
C\mathscr{H}^{\mathbf{n}}\left(M_{0}\cap\left(Y^{\mathbf{n}}\times B^1\right)\right).
\end{align}
For $x\in M_{s_1}\cap(E_1^{\mathbf{n}}\times B^1)$ we can estimate
$\zeta(x)\eta(s_1,x)-2\mathbf{n}\lambda^{-2\mathbf{n}}s_1\geq 1$,
where we used \eqref{becominggraphthm54}, $s_1=\lambda^{3\mathbf{n}}$, $\lambda\leq\lambda_0$ and $\lambda_0$ small.
Hence \eqref{becominggraphthm52} implies \eqref{becominggraphthm51}.

Let $C_L\in (1,\infty)$ be a constant depending on $L$,
which may increase in each step.
By $\sup|Df|\leq L$ 
we have $JF\leq C_L$.
Then in view of assumption \eqref{becominggraphthmb}
we obtain
\begin{align}
\label{becominggraphthm55}
\begin{split}
\mathscr{H}^{\mathbf{n}}\left(M_{0}\cap\left((Y^{\mathbf{n}}\cap D^{\mathbf{n}})\times B^1\right)\right)
&\leq 
C_L\mathscr{L}^{\mathbf{n}}(Y^{\mathbf{n}})
\leq 
C_L\lambda^{\mathbf{n}}\rho^\mathbf{n}.
\end{split}
\end{align}
On the other hand as $Y^{\mathbf{n}}\subset\mathbf{B}^{\mathbf{n}}(0,8)$ 
we see $Y^{\mathbf{n}}\setminus D^{\mathbf{n}}\subset E^{\mathbf{n}}$.
Thus by assumption \eqref{becominggraphthmc2}
we can estimate
\begin{align}
\label{becominggraphthm56} 
\mathscr{H}^{\mathbf{n}}\left(M_{0}\cap\left((Y^{\mathbf{n}}\setminus D^{\mathbf{n}})\times B^1\right)\right)
\leq
\lambda^{\mathbf{n}}\rho^{\mathbf{n}}.
\end{align}
Combining \eqref{becominggraphthm51} with \eqref{becominggraphthm55} and \eqref{becominggraphthm56}
we conclude
\begin{align}
\label{becominggraphthm57} 
\mathscr{H}^n\left(M_{s_1}\cap\left(E_1^{\mathbf{n}}\times B^1\right)\right)
\leq
C_L\lambda^{\mathbf{n}}\rho^{\mathbf{n}}.
\end{align}

Set $\beta:=\sqrt{t_0-s_1}
=\sqrt{\Lambda_0\lambda^2-\lambda^{3\mathbf{n}}}$
and $\rho_0:=2^{-3}$.
In particular
\begin{align}
\label{becominggraphthm41} 
c\Lambda_0\lambda^2
\leq\beta^2
\leq\Lambda_0\lambda^2
\leq\beta_0^2,
\end{align}
and $I=[t_0,t_0+16\sigma_0]\subset (s_1+\beta^2\rho_0^2,s_1+\beta_0^2\rho_0^2)$,
where $\beta_0$ is the constant from Proposition \ref{flatbecomegraphprop}.
Here we used $\lambda\leq\lambda_0$,
$\lambda_0$ small enough depending on $\Lambda_0$
and $\sigma_0$ small depedning on $\beta_0$.

Let $z=(\hat{z},0)\in\mathbf{B}^{\mathbf{n}}(0,4)\times\{0\}$ be arbitrary.
We want to use Proposition \ref{flatbecomegraphprop}
with $z_0=z$, $s_2=t_2$ 
and $D_0=\mathbf{B}^{\mathbf{n}}(\hat{z},2^{-2})\setminus E_1^{\mathbf{n}}$.
In view of \eqref{becominggraphthm41} and for $\Lambda_0$ large enough,
the slab statement \eqref{becominggraphthm21} with $r(s_1)\leq C\lambda^2$ implies \eqref{flatbecomegraphpropa1}.
By choice of $D_0$ and \eqref{becominggraphthm33} also \eqref{flatbecomegraphpropb1} holds.
Choosing $\lambda_0$ small enough depending on $\Lambda_1$ of Lemma \ref{flatgraphlem} for $L_0=L$,
we see that \eqref{becominggraphthm32} implies \eqref{flatbecomegraphpropb2}.
Choosing $\Lambda_0$ large depending on $L$ we can use
\eqref{becominggraphthm41} and \eqref{becominggraphthm57} to obtain \eqref{flatbecomegraphpropb3}.
At last we consider the points $\hat{a}_0^{\pm}:=\hat{z}\pm 2^{-4}\hat{\mathbf{e}}_{\mathbf{n}}$.
Then
\begin{align*}
|(\hat{a}_0^{+}-\hat{a}_0^{-})\cdot\hat{\mathbf{e}}_{\mathbf{n}}|=2^{-3}
> 4^{\mathbf{n}+1}\lambda^\mathbf{n}+2^{-3}\sqrt{\beta},
\end{align*}
where we used \eqref{becominggraphthm41}, $\lambda\leq\lambda_0$ 
and chose $\lambda_0$ small enough.
Now by definition of $E_1^{\mathbf{n}}$ we see that \eqref{flatbecomegraphpropb4} holds
for $\hat{a}_0=\hat{a}_0^{+}$ or $\hat{a}_0^{-}$.
Proposition \ref{flatbecomegraphprop} then yields a
$g_{z}:I\times\mathbf{B}^{\mathbf{n}}(z,2^{-3})\to\mathbb{R}$
with
\begin{align}
\label{becominggraphthm61} 
M_{t}\cap\mathbf{C}(z,2^{-3},2^{-3})
=\mathrm{graph}(g_{z}(t,\cdot)).
\end{align}
Note that for $t\in I$ we have
$t\geq t-s_1\geq\frac{t}{2}$
and $r(t)\leq 2^{-4}$,
where we estimated $\lambda\leq\lambda_0$ 
and we chose $\lambda_0$ small enough.
Hence \eqref{flatbecomegraphpropd} implies \eqref{becominggraphthme} with $g$ replaced by $g_{z}$
and by \eqref{becominggraphthm22} we see,
that \eqref{becominggraphthm61} actually holds for the larger cylinder $\mathbf{C}(z,2^{-3},4)$.
As $z\in\mathbf{B}^{\mathbf{n}}(0,4)\times\{0\}$ was arbitrary,
the $g_{z}$ can be combined to obtain the desired $g$.
\end{proof}
%
%
%
%
%
%
%
%
%
Theorem \ref{becominggraphthm} naturaly implies a global result.
\begin{cor}
\label{globalbecominggraphcor}
For every $L\in [1,\infty)$ 
there exists a $\Lambda_3\in (1,\infty)$
such that the following holds:
Let $\varrho\in (0,\infty)$, $K\in [1,\infty)$, $t_0=\Lambda_3K^2\varrho^2$, 
$T\in(t_0,\infty]$
and let $(M_t)_{t\in [0,T)}$ be a \mcf\; in $\mathbb{R}^{\mathbf{n}+1}$.
Assume
\begin{align}
\label{globalbecominggraphcora} 
M_{0}\subset \mathbb{R}^{\mathbf{n}}\times\chball{\varrho}.
\end{align}
Suppose there exists a closed subset 
$E_0\subset\mathbb{R}^{\mathbf{n}-1}\times\chball{\varrho}$ 
and a function
$f:\mathbb{R}^{\mathbf{n}}\setminus E_0\to\mathbb{R}$ 
with $\sup|Df|\leq L$ and
\begin{align}
\label{globalbecominggraphcorb} 
M_{0}\cap\left((\mathbb{R}^{\mathbf{n}}\setminus E_0)\times\mathbb{R}\right)
=\mathrm{graph}(f).
\end{align}
Also suppose
\begin{align}
\label{globalbecominggraphcorc} 
\mathscr{H}^{\mathbf{n}}\left(M_{0}\cap\left(\left(E_0\cap\mathbf{B}^{\mathbf{n}}(\hat{a},4\varrho)\right)\times\mathbb{R}\right)\right)
\leq K\varrho^{\mathbf{n}}
\end{align}
for all $\hat{a}\in\mathbb{R}^{\mathbf{n}-1}\times\{0\}$.
Then there exists a 
$g:[t_0,\infty)\times \mathbb{R}^{\mathbf{n}}\to\mathbb{R}$
with
\begin{align}
\label{globalbecominggraphcord}
M_{t}=\mathrm{graph}(g(t,\cdot))
\;\;\;\forall t\in [t_0,T)
\end{align}
\end{cor}
%
%
%
%
%
\begin{proof}
We may assume $\varrho=1$.
First note that by the comparison principle 
and \eqref{globalbecominggraphcora} we have
$M_{t}\subset \mathbb{R}^{\mathbf{n}}\times\chball{1}$
for all $t\in[0,T)$. 
This also follows from Corollary \ref{heightboundcor} with $R\to\infty$.
Choose $\sigma_0$, $\Lambda_0=\Lambda_0(L)$ and $\lambda_0=\lambda_0(L)$ according to 
Theorem \ref{becominggraphthm}.
Set 
\begin{align*}
\lambda_1^2:=\min\{\lambda_0^2,(2\Lambda_0)^{-1}\sigma_0\},
\;\;\;
\rho:=C_0\lambda_1^{-2\mathbf{n}}K
\end{align*}
for some constant $C_0\in (1,\infty)$ which will be chosen below.

Let $\hat{x}\in\mathbb{R}^{\mathbf{n}}$ be arbitrary
and set $A:=\mathbf{B}^{\mathbf{n}}(\hat{x},5\rho)\cap \left(\mathbb{Z}^{\mathbf{n}-1}\times\{0\}\right)$.
Then $\sharp A\leq C\rho^{\mathbf{n}-1}$
and $\mathbf{B}^{\mathbf{n}}(\hat{x},4\rho)\cap E_0\subset\bigcup_{\hat{a}\in A}\mathbf{B}^{\mathbf{n}}(\hat{a},4)$.
Hence by \eqref{globalbecominggraphcorc} we have
\begin{align*}
\mathscr{H}^{\mathbf{n}}\left(M_{0}
\cap\left(\left(E_0\cap \mathbf{B}^{\mathbf{n}}(\hat{x},4\rho)\right)\times\mathbb{R}\right)\right)
\leq C\rho^{\mathbf{n}-1}K
\leq\lambda_1^{2\mathbf{n}}\rho^{\mathbf{n}},
\end{align*}
where we used $\lambda_1\leq 1$ and chose $C_0$ large enough.
Using Theorem \ref{becominggraphthm}
with $t_1=0$, $t_2=\lambda_0^2\rho^2$, $x_0=(\hat{x},0)$, $\lambda=\lambda_1$, $h=0$
and $D=\mathbf{B}^{\mathbf{n}}(\hat{x},4\rho)\setminus E_0$
yields a 
$g_{\hat{x}}:\mathbf{B}^{\mathbf{n}}(\hat{x},\rho)\to\mathbb{R}$
with
\begin{align*}
M_{t_0}\cap\mathbf{C}(\hat{x},\rho,\rho)
=\mathrm{graph}(g_{\hat{x}}).
\end{align*}
Here we used $t_0=\Lambda_3 K^2=\Lambda_0\lambda_1^2\rho^2$
for $\Lambda_3=C_0^2\Lambda_0\lambda_1^{2-4\mathbf{n}}$.
As $\rho\geq 1$ and $\hat{x}\in\mathbb{R}^n$ was arbitrary, 
we can combine the $g_{\hat{x}}$
to obtain a 
$g_{t_0}:\mathbb{R}^{\mathbf{n}}\to\mathbb{R}$
with
\begin{align*}
M_{t_0}=\mathrm{graph}(g_{t_0}).
\end{align*}
Now by \cite[5.1]{eckerh1} 
we know there exists a unique graphical solution for all times,
which establishes the result.
\end{proof}

%% file: appendix.tex
This appendix contains some details on the relation between height, tilt, curvature and graphical representation
of a manifold $M$.
In particular we show,
how curvature bounds combined with a slab condition yield graphical representability.
The main ingridient is the following result from Colding and Minicozzi.
%
%
%
%
\begin{secprop}[{\cite[2.4]{colding1}}]
\label{flatparaprop}
There exist $C\in (1,\infty)$ and $\alpha_0\in (0,1)$ such that the following holds:
Let $\alpha\in (0,\alpha_0]$, $R\in (0,\infty)$
and let $M$ be a submanifold of $\mathbb{R}^{\mathbf{n}+\mathbf{k}}$
with $0\in M$ and $\mathbf{T}(M,0)=\mathbb{R}^{\mathbf{n}}\times\{0\}^{\mathbf{k}}$. 
Suppose
\begin{align}
\label{flatparapropa2}
\partial M\cap\mathbf{B}(0,2R)=\emptyset,
\\
\label{flatparapropa1}
\sup_{x\in M\cap\mathbf{B}(0,2R)}\left|\mathbf{A}(M,x)\right|\leq\alpha R^{-1}.
\end{align}
Then there exists a $g\in\mathcal{C}^{2}\left(\mathbf{B}^{\mathbf{n}}(0,2R),\mathbb{R}^{\mathbf{k}}\right)$ with
$\mathrm{graph}(g)\subset M$ and
\begin{align} 
\label{flatparapropc}
g(0)=0,
\;\;\;\; Dg(0)=0,
\\
\label{flatparapropd}
\max\{R^{-1}\sup|g|,\sup|Dg|,R\sup|D^2g|\}\leq C\alpha.
\end{align}
\end{secprop}

\begin{seccor}
\label{flatparacor}
There exist $C\in (1,\infty)$ and $\alpha_0\in (0,1)$ such that the following holds:
Let $\alpha\in (0,\alpha_0]$, $R\in (0,\infty)$, $L\in[1,\infty)$ $\hat{a}\in\mathbb{R}^{\mathbf{n}}$.
Let $f\in\mathcal{C}^{2}\left(\mathbf{B}^{\mathbf{n}}(\hat{a},2R),\mathbb{R}^{\mathbf{k}}\right)$ 
with $\sup|Df|\leq L$.
Set $M:=\mathrm{graph}(f)$ and assume
\begin{align}
\label{flatparacora}
\sup_{x\in M}\left|\mathbf{A}(M,x)\right|\leq\alpha R^{-1}.
\end{align}
Set $a:=(\hat{a},f(\hat{a}))$
and let $S\in\mathbf{SO}(\mathbf{n}+\mathbf{k})$ be such that 
$S[\mathbb{R}^{\mathbf{n}}\times\{0\}^{\mathbf{k}}]=\mathbf{T}(M,a)$.
Then there exists a local parametrisation 
$g\in\mathcal{C}^{2}\left(\mathbf{B}^{\mathbf{n}}(0,R),\mathbb{R}^{\mathbf{k}}\right)$ with 
\begin{align}
\label{flatparacorb}
M\cap\left(\mathbf{B}^{\mathbf{n}}(\hat{a},(2L)^{-1} R)\times\mathbb{R}^{\mathbf{k}}\right)
\subset S(\mathrm{graph}(g))+a
\subset M.
\end{align}
Furthermore $g$ satisfies
\begin{align}
\label{flatparacord}
g(0)=0,
\;\;\;\; Dg(0)=0,
\\
\label{flatparacorc}
\max\left\{R^{-1}\sup|g|+\sup|Dg|+R\sup|D^2g|\right\}\leq C\alpha.
\end{align}
\end{seccor}
%
%
\begin{proof}
Let $\alpha_0$ be from Proposition \ref{flatparaprop}.
As $\alpha\leq\alpha_0$ we can apply
Proposition \ref{flatparaprop} to the manifold $S^{-1}[M-a]$
to obtain a function $g\in\mathcal{C}^{2}\left(\mathbf{B}^{\mathbf{n}}(0,R),\mathbb{R}^{\mathbf{k}}\right)$,
which satisfies \eqref{flatparacord}, \eqref{flatparacorc} 
and $\mathrm{graph}(g)\subset S^{-1}[M-a]$.

Set $N:=S[\mathrm{graph}(g)]+a.$
We already know $N\subset M$
and want to show $M\cap\left(\mathbf{B}^{\mathbf{n}}(\hat{a},(2L)^{-1} R)\times\mathbb{R}^{\mathbf{k}}\right)\subset N$.
It suffices to prove
\begin{align}
\label{flatparacor11}
M\cap\left(\mathbf{B}^{\mathbf{n}}(\hat{a},(2L)^{-1} R)\times\mathbb{R}^{\mathbf{k}}\right)\cap\partial N=\emptyset.
\end{align}
Suppose this is false and there exists a point $y_0$ in this set.
Then we find sequences $(\hat{w}_m)_{m\in\mathbb{N}}$ and $(\hat{y}_m)_{m\in\mathbb{N}}$
with $\hat{w}_m\in\mathbf{B}^{\mathbf{n}}(0,R)$, $\hat{y}_m\in\mathbf{B}^{\mathbf{n}}(\hat{a},(2L)^{-1} R)$,
\begin{align*}
S(\hat{w}_m,g(\hat{w}_m))+a=(\hat{y}_m,f(\hat{y}_m))=:y_m
\end{align*}
for all $m\in\mathbb{N}$ and such that $\lim_{m\to\infty} y_m=y_0$, $\lim_{m\to\infty}|\hat{w}_m|=R$.
But for $|\hat{w}_m|$ we can estimate
\begin{align*}
|\hat{w}_m|^2
&\leq
|S(\hat{w}_m,g(\hat{w}_m))|^2
=
|y_m-a|^2
=
|\hat{y}_m-\hat{a}|^2
+ |f(\hat{y}_m)-f(\hat{a})|^2
\\&
\leq 
(2L)^{-2}R^2 + L^2|\hat{y}_m-\hat{a}|^2
\leq 
(1+L^2)(2L)^{-2}R^2
=\frac{R^2}{2}
\end{align*}
for all $m\in\mathbb{N}$.
Here we used $a=(\hat{a},f(\hat{a}))$, $\sup|Df|\leq L$ and $L\geq 1$.
Thus we obtain a contradiction.
This proves \eqref{flatparacor11} which establishes the result.
\end{proof}

%
%
%
%
We observe that, if for a family of submanifolds $(M_t)$ 
the tangent space $\mathbf{T}(M_{t},x)$ is never perpendicular to $\mathbb{R}^\mathbf{n}$,
then $(M_t)$ consists of a constant number of sheets.
%
%
\begin{seclem}
\label{constsheatlem}
Let $r,\Gamma_0\in (0,\infty)$, $\kappa_0\in [0,1)$, $t_1\in\mathbb{R}$, $t_2\in (t_1,\infty)$.
Let $M$ be a submanifold of $\mathbb{R}^{\mathbf{n}+\mathbf{k}}$
and consider 
$\Psi\in\mathcal{C}([t_1,t_2)\times M,\mathbb{R}^{\mathbf{n}+\mathbf{k}})$ 
such that
$\Psi_t:=\Psi(t,\cdot)\in\mathcal{C}^{1}(M,\mathbb{R}^{\mathbf{n}+\mathbf{k}})$ 
are embeddings for all $t\in [t_1,t_2)$. 
Set $M_t:=\Psi_t[M]$.
Suppose
\begin{align}
\label{constsheatlemc} 
\|\mathbf{T}(M_{t},x)_{\natural}-(\mathbb{R}^{\mathbf{n}}\times\{0\}^{\mathbf{k}})_{\natural}\| \leq\kappa_0<1
\end{align}
for all $t\in [t_1,t_2)$ and all $x\in M_t\cap\mathbf{C}(0,r,\Gamma_0)$.
Also suppose
\begin{align}
\label{constsheatlema}
\begin{split}
\partial M_t\cap\mathbf{B}^{\mathbf{n}}(0,r)\times\overline{\mathbf{B}^{\mathbf{k}}(0,\Gamma_0)}=\emptyset,
\\
M_t\cap\mathbf{B}^{\mathbf{n}}(0,r)\times\partial\mathbf{B}^{\mathbf{k}}(0,\Gamma_0)=\emptyset
\end{split}
\end{align}
for all $t\in [t_1,t_2)$.
Then there exists an $m_0\in\mathbb{N}\cup\{0\}$ such that
\begin{align}
\label{constsheatlemd} 
\sharp\left\{x\in M_t\cap\{\hat{y}\}\times\mathbf{B}^{\mathbf{k}}(0,\Gamma_0)\right\}=m_0
\end{align}
for all $\hat{y}\in\mathbf{B}^{\mathbf{n}}(0,\rho)$ and all $t\in [t_1,t_2)$.
\end{seclem}
%
%
\begin{proof}
Fix $t\in [t_1,t_2)$.
For $\hat{x}\in\mathbf{B}^{\mathbf{n}}(0,r)$ set
\begin{align*}
A(\hat{x}):=\left\{a\in M_t\cap\{\hat{x}\}\times\mathbf{B}^{\mathbf{k}}(0,\Gamma_0)\right\},
\;\;\;
N(\hat{x}):=\sharp A(\hat{x}).
\end{align*}
Note that by \eqref{constsheatlemc} 
we can use the implicit function theorem
to see that the points in $A(\hat{x})$ are discrete.
Thus, by \eqref{constsheatlema} 
and as $\mathbf{B}^{\mathbf{k}}(0,\Gamma_0)$ is bounded
we observe that $N(\hat{x})$ has to be finite.

Consider arbitrary $\hat{y}\in \mathbf{B}^{\mathbf{n}}(0,r)$.
We want to show there exists some $\delta=\delta(\hat{y})$ such that
\begin{align}
\label{constsheatlem11}
N(\hat{x})=N(\hat{y})=:N
\end{align}
for all $\hat{x}\in\mathbf{B}^{\mathbf{n}}(\hat{y},\delta)$.
There exist $N$ 
different $a_i=(\hat{y},h_i)\in A(\hat{y})$.
In view of \eqref{constsheatlemc}
we can use the implicit function theorem,
to obtain a $\delta_1\in (0,r)$ and
$g_i\in\mathcal{C}^{2}\left(\mathbf{B}^{\mathbf{n}}(\hat{x}_i,2\delta_1),\mathbb{R}^{\mathbf{k}}\right)$
with $g_i(\hat{y})=h_i$,
such that 
\begin{align*}
M_t\cap\mathbf{B}(a_i,2\delta_1)\subset\mathrm{graph}(g_i)
\subset M_t\cap\mathbf{C}(0,r,\Gamma_0)
\end{align*}
for all $i=1,\ldots,N$.
Note that $\delta_1$ can be chosen 
such that the $\mathrm{graph}(g_i)$ are disjoint.
In particular this shows "$\geq$" in \eqref{constsheatlem11} for $\delta\leq\delta_1$.
Let
$W:=(\{\hat{y}\}\times\overline{\mathbf{B}^{\mathbf{k}}(0,\Gamma_0)})
\setminus\bigcup_{i=1}^{N}\mathbf{B}(a_i,\delta_1)$,
so $W\cap M_t=\emptyset$. 
Then by \eqref{constsheatlema} and compactness of $W$ 
there exists a $\delta\in (0,\delta_1]$ such that "$\leq$" holds in \eqref{constsheatlem11}, which proves \eqref{constsheatlem11}.

Set $m(t):=N(0)$.
By \eqref{constsheatlem11} we see $\{N(\hat{y})=m(t)\}$ is open and closed in $\mathbf{B}^{\mathbf{n}}(0,r)$.
As $\mathbf{B}^{\mathbf{n}}(0,r)$ is connected this shows \eqref{constsheatlemd}
for fixed $t$.
By continuity in time, this establishes the result.
\end{proof}

Combining this result with Proposition \ref{flatparaprop} yields
%
%
\begin{seclem}
\label{multigraphlem}
There exist $C\in (1,\infty)$ and $\alpha_1\in (0,1)$ such that the following holds:
Let $K,r\in (0,\infty)$, $\xi\in (0,K]$
and let $M$ be a submanifold of $\mathbb{R}^{\mathbf{n}+\mathbf{k}}$.
Assume $K\xi\leq\alpha_1$
\begin{align}
\label{multigraphlema1}
M\cap\mathbf{C}(0,2r,2r)
\subset\mathbf{C}(0,2r,\xi^2 r),
\\
\label{multigraphlema2}
\partial M\cap\mathbf{C}(0,2r,2r)=\emptyset.
\end{align}
Also suppose
\begin{align}
\label{multigraphlemb}
\sup_{x\in M\cap\mathbf{C}(0,2r,2r)}\left|\mathbf{A}(M,x)\right|\leq K^2 r^{-1},
\end{align}
Then there exists an $m_0\in\mathbb{N}\cup\{0\}$ 
and $g_i\in\mathcal{C}^{2}\left(\mathbf{B}^{\mathbf{n}}(0,r),\mathbb{R}^{\mathbf{k}}\right)$
such that
\begin{align}
\label{multigraphlemd} 
M\cap\mathbf{C}(0,r,r)=\bigcup_{i=1}^{m_0}\mathrm{graph}(g_i),
\end{align}
where the union on the right is disjoint
and we set $\bigcup_{i=1}^{0}G_i:=\emptyset$.
Moreover we have
\begin{align}
\label{multigraphleme}
\sup|g_i|\leq C\xi^2 r,
\;\;\;
\sup|Dg_i|\leq CK\xi,
\;\;\;
\sup|D^2g_i|\leq CK^2 r^{-2}
\end{align}
for all $i=1,\ldots,m_0$.
\end{seclem}
%
%
\begin{proof}
We may assume $r=1$.
Let $a\in M\cap\mathbf{C}(0,1,1)$ be arbitrary.
We want to show
\begin{align}
\label{multigraphlem11}
\left\|(\mathbb{R}^{\mathbf{n}}\times\{0\}^{\mathbf{k}})_{\natural}-\mathbf{T}(M,a)_{\natural}\right\|\leq CK\xi.
\end{align}
Consider $S\in\mathbf{SO}(\mathbf{n}+\mathbf{k})$ with $S(\mathbb{R}^{\mathbf{n}}\times\{0\}^{\mathbf{k}})=\mathbf{T}(M,a)$.
Set $\tau_i:=S(\mathbf{e}_{i})$, $i\in\{1,\ldots,\mathbf{n}\}$.
For $\xi_0$ small enough we can apply Proposition \ref{flatparaprop}
with $R:=(2K)^{-1}\xi$ and $\alpha=K\xi$ to the manifold $S^{-1}[M-a]$,
to obtain a function $g\in\mathcal{C}^{2}(\mathbf{B}^{\mathbf{n}}(0,2R),\mathbb{R}^{\mathbf{k}})$
with $g(0)=0$, $\sup|g|\leq C\xi^2$ 
and $S[\mathrm{graph}(g)+a]\subset M$.
Here we used $K\xi\leq\alpha_1$ and chose $\alpha_1\leq\alpha_0$.

For $i\in\{1,\ldots,\mathbf{n}\}$
define $v_i\in\mathrm{graph}(g)$ and $w_i\in\mathbb{R}^{\mathbf{n}}\times\{0\}^{\mathbf{k}}$ by
\begin{align*}
v_i&:=\left(R\hat{\mathbf{e}}_i,g(R\hat{\mathbf{e}}_i)\right),
\;\;\;
w_i
:=
Sv_i
-\left(\{0\}^{\mathbf{n}}\times\mathbb{R}^{\mathbf{k}}\right)_{\natural}
\left(Sv_i\right)
\end{align*}
Using $S(\mathbf{e}_{i})=\tau_i$ and \eqref{multigraphlema1}
we can estimate
\begin{align*}
|R\tau_i-w_i|
&\leq
|S\left(0,g(R\hat{\mathbf{e}}_i)\right)|
+|\left(\{0\}^{\mathbf{n}}\times\mathbb{R}^{\mathbf{k}}\right)_{\natural}
\left(Sv_i\right)|
\leq C\xi^2.
\end{align*}
Now set $\tilde{w}_i:=Rw_i$ for $i\in\{1,\ldots,\mathbf{n}\}$.
As $R=(2K)^{-1}\xi$, this yields $|\tau_i-\tilde{w}_i|\leq CK\xi$.
The $(\tilde{w}_i)$ lie in $\mathbb{R}^{\mathbf{n}}\times\{0\}^{\mathbf{k}}$
and the $(\tau_i)$ form an orthonormal basis of $\mathbf{T}(M,0)$,
so this establishes \eqref{multigraphlem11} (see \cite[8.9(5)]{allard}).

In view of \eqref{multigraphlem11} and $CK\xi\leq C\alpha_1<1$,
we can use Lemma \ref{constsheatlem} 
with $M_t\equiv M$ and $\Gamma_0=r$ to obtain the $g_i$. For the bounds in \eqref{multigraphleme} use
\eqref{multigraphlema1}, \eqref{multigraphlemb} and \eqref{multigraphlem11}
combined with \eqref{tiltderivativerema} and \eqref{secderivativerema}.
\end{proof}

%
%
%
%
We end the appendix by giving an alternative condition to
\eqref{flatbecomegraphpropb3} or \eqref{becominggraphthmc2}.
%
%
\begin{seclem}
\label{pcurveboundlem}
There exists a $C\in (1,\infty)$
such that the following holds:
Let $p\in (\mathbf{n},\infty)$, $\varrho,K\in (0,\infty)$, $\gamma,\delta\in (0,1]$, $y_0\in\mathbb{R}^{\mathbf{n}+\mathbf{k}}$
and let $M$ be a submanifold of $\mathbb{R}^{\mathbf{n}+\mathbf{k}}$.
Assume $\partial M\cap \mathbf{B}(y_0,2\varrho)=\emptyset$,
\begin{align}
\label{pcurveboundlema}
M\cap\mathbf{B}(y_0,2\varrho)
\subset\mathbf{C}(y_0,2\varrho,\gamma\varrho)
\end{align}
and
\begin{align}
\label{pcurveboundlemb}
\varrho^{-\mathbf{n}}\mathscr{H}^{\mathbf{n}}\left(M\cap\mathbf{B}(y_0,2\varrho)\right)
+\varrho^{p-\mathbf{n}}\int_{\mathbf{B}(y_0,2\varrho)}|\mathbf{H}(M,x)|^pd\mu_M(x)
\leq K.
\end{align}
Then
$\mathscr{H}^{\mathbf{n}}\left(M\cap\mathbf{C}(y_0,\delta\varrho,\varrho)\right)
\leq C 2^p(p-\mathbf{n})^{-1}K(\gamma+\delta)^{\mathbf{k}}\delta^{\mathbf{n}-\mathbf{k}}\varrho^{\mathbf{n}}$.
\end{seclem}
\begin{proof}
We may assume $y_0=0$ and $\varrho=2$.
Set $A:=\mathbf{C}(0,2\delta,2\gamma)\cap \delta\mathbb{Z}^{\mathbf{n}+\mathbf{k}}$.
Let $a\in A$, then $\mathbf{B}(a,1)\subset\mathbf{B}(0,4)$.
By \eqref{pcurveboundlemb} and Simon's monotonicity formula \cite[4.3.2]{simon}
with $R=1$, $r=2\delta$ and $\Gamma^p=K$
we can estimate
\begin{align*}
(2\delta)^{-\mathbf{n}}\mu_M\left(\mathbf{B}(a,2\delta)\right)
&\leq
\left(\left(\mu_M\left(\mathbf{B}(a,1)\right)\right)^{\frac{1}{p}}
+\frac{K^{\frac{1}{p}}(1-\delta^{\frac{p-\mathbf{n}}{p}})}{p-\mathbf{n}}\right)^p
\leq
\frac{2^pK}{p-\mathbf{n}}.
\end{align*}
By assumption \eqref{pcurveboundlema} we have 
$M\cap\mathbf{C}(0,\delta,2)
\subset\bigcup_{a\in A}\mathbf{B}(a,2\delta)$
and also $\sharp A\leq C(\gamma\delta^{-1}+1)^{\mathbf{k}}$.
This establishes the result.
\end{proof}

%% file: almost_graphical_mcf.bbl
\begin{thebibliography}{10}
%
%
%
\bibitem[All72]{allard} 
W. K. Allard,  
\emph{On the first variation of a varifold}, 
Annals of Math. \textbf{95}, 417-491 (1972)
%
%
%
\bibitem[Bra78]{brakke} 
K. A. Brakke,
\emph{The Motion of a Surface by its Mean Curvature}, 
Math. Notes Princeton, NJ, Princeton University Press (1978)
%
%
%
\bibitem[BH12]{brendleh} 
S. Brendle and G. Huisken,
\emph{Mean curvature flow with surgery of mean convex surfaces in $\mathbb{R}^3$}, 
arXiv:1309.1461 (2012)
%
%
%
\bibitem[CY07]{chenyin}
B.-L. Chen and L. Yin,
\emph{Uniqueness and pseudo locality theorems of the mean curvature flow},
Comm. Anal. Geom. \textbf{15}, no.3, 435-490 (2007)
%
%
%
\bibitem[CM11]{colding1}
T. H. Colding and W. P. Minicozzi,
\emph{A Course in Minimal Surfaces}
American Mathematical Society (2011).
%
%
%
\bibitem[Eck04]{eckerb} 
K. Ecker,
\emph{Regularity Theory for Mean Curvature Flow}, 
Birkh\"auser Verlag, Boston-Basel-Berlin (2004)
%
%
%
\bibitem[EH89]{eckerh2}
K. Ecker and G. Huisken,
\emph{Mean curvature evolution by entire graphs}, 
Annals of Math. \textbf{139}, 453-471 (1989)
%
%
%
\bibitem[EH91]{eckerh1}
K. Ecker and G. Huisken, 
\emph{Interior estimates for hypersurfaces moving by mean curvature}, 
Invent. Math. \textbf{105}, 547-569 (1991)
%
%
%
\bibitem[Gra87]{grayson} 
M. A. Grayson,
\emph{The heat equation shrinks embedded plane curves to round points}, 
J. Differential Geometry \textbf{26}, 285-314 (1987)
%
%
%
\bibitem[Hui84]{huisken1} 
G. Huisken,
\emph{Flow by mean curvature of convex surfaces into spheres}, 
J. Differential Geometry \textbf{20}, 237-266 (1984)
%
%
%
\bibitem[Hui90]{huisken2} 
G. Huisken, 
\emph{Asymptotic behaviour for singularities of the mean curvature flow}, 
J. Differential Geometry \textbf{31}, 285-299 (1990)
%
%
%
\bibitem[INS14]{ilnesch}
T. Ilmanen, A. Neves and F. Schulze
\emph{On short time existence for the planar network flow},
arXiv:1407.4756 (2014)
%
%
%
\bibitem[LO77]{lawsonO}
H. B. Lawson and P. Osserman,
\emph{Non-existence, non-uniqueness and irregularity of solutions to the minimal surface system},
Acta Math. \textbf{139}, no. 1-2, 1-17 (1977)
%
%
%
\bibitem[Lah14]{lahiri}
A. Lahiri
\emph{Regularity of the Brakke flow},
Dissertation, Freie Universität Berlin (2014)
%
%
%
\bibitem[Sim83]{simon} 
L. M. Simon,
\emph{Lectures on geometric measure theory}, 
Proc. of the CMA, Vol.3 (1983)
%
%
%
\bibitem[Wan04]{wang} 
M. T. Wang
\emph{The mean curvature flow smoothes Lipschitz submanifolds}, 
Comm. Anal. Geom. \textbf{12}(3), 581-599 (2004)
%
%
%
\bibitem[Whi05]{white} 
B. White, 
\emph{A local regularity theorem for mean curvature flow}, 
Ann. of Math. \textbf{161}, 148-1519 (2005)
%
%
%
\end{thebibliography}
